\documentclass{amsart}
\usepackage{amsmath, amscd, amssymb, amsthm}
\usepackage{bbm}
\usepackage{latexsym}
\usepackage{amsfonts}
\usepackage{graphicx}
\usepackage[all,cmtip]{xy}
\usepackage[colorlinks,linkcolor=blue,breaklinks=blue,urlcolor=blue,citecolor=blue,anchorcolor=blue,pagebackref]{hyperref}%
\usepackage{geometry}
\newtheorem{theorem}{Theorem}
\newtheorem{lemma}{Lemma}
\newtheorem{corollary}[theorem]{Corollary}
\newtheorem{remark}{Remark}

\newtheorem*{definition}{Definition}
\newtheorem{conjecture}{Conjecture}

\renewcommand*\backref[1]{}
\renewcommand*\backrefalt[4]{ \ifcase #1 \or (cited on page #2) \else (cited on pages #2) \fi}

\newcommand{\be}{\begin{equation}}
\newcommand{\ee}{\end{equation}}
\newcommand{\bea}{\begin{eqnarray}}
\newcommand{\eea}{\end{eqnarray}}

\newcommand{\vs}{\vspace{0.5cm}}

\def\XXint#1#2#3{{\setbox0=\hbox{$#1{#2#3}{\int}$ }
\vcenter{\hbox{$#2#3$ }}\kern-.6\wd0}}

\begin{document}

\title[Fino-Vezzoni conjecture on Lie algebras with abelian ideals of codimension two]{Fino-Vezzoni conjecture on Lie algebras with abelian ideals of codimension two}

\author{Kexiang Cao}
\address{Kexiang Cao. School of Mathematical Sciences, Chongqing Normal University, Chongqing 401331, China}
\email{{ caokx1214@qq.com}}\thanks{The corresponding author Zheng is partially supported by National Natural Science Foundations of China
with the grant No.12071050 and 12141101, Chongqing grant cstc2021ycjh-bgzxm0139 and CQYC2021059145, and is supported by the 111 Project D21024.}

\author{Fangyang Zheng}
\address{Fangyang Zheng. School of Mathematical Sciences, Chongqing Normal University, Chongqing 401331, China}
\email{20190045@cqnu.edu.cn; franciszheng@yahoo.com} \thanks{}


\begin{abstract}
In this paper, we confirm the Fino-Vezzoni Conjecture for unimodular Lie algebras which contain abelian ideals of codimension two, a natural generalization to the class of almost abelian Lie algebras. This provides new evidence towards the validity of the conjecture on a very special type of $3$-step solvmanifolds. 
\end{abstract}

\maketitle

\tableofcontents

\section{Introduction and statement of results}\label{intro}

An important open question in non-K\"ahler geometry is the following interesting conjecture raised by Fino and Vezzoni (\cite{FV15, FV16}):
\begin{conjecture}[{\bf Fino-Vezzoni}] \label{conj1}
Let $M^n$ be a compact complex manifold. If $M^n$ admits a balanced Hermitian metric and a pluriclosed Hermitian metric, then it must admit a K\"ahler metric. 
\end{conjecture}

Recall that a Hermitian metric $g$ on $M^n$ is said to be {\em balanced} if $d(\omega^{n-1})=0$, and $g$ is said to be {\em pluriclosed} if $\partial \overline{\partial} \omega =0$.  Here $\omega$ is the K\"ahler form of $g$. Balanced metrics and pluriclosed metrics are natural extensions of K\"ahler metrics, and they are two of the special Hermitian structures that are extensively studied. See \cite{FGV, Fu, LU, Paradiso, ZZ-JGP} for a sample of discussion on balanced metrics and see \cite{AL, EFV, FinoTomassini, ST, STian} for  pluriclosed metrics.

The aforementioned conjecture simply says that these two types of special Hermitian metrics cannot coexist on any compact complex manifold, unless the manifold is K\"ahlerian (i.e., admitting a  K\"ahler metric). 

A few quick remarks about the conjecture. The first one is that the dimension $n$ starts with $3$ here,  as in $n=2$ the balanced condition coincides with the K\"ahler condition, which is $d\omega =0$. Also, the compactness assumption in the conjecture is necessary, and there are counterexamples in the non-compact case. Finally, if a Hermitian metric $g$ is both balanced and pluriclosed, then it must be K\"ahler by a theorem of Alexandrov and Ivanov \cite{AI}. So the main difficulty of the conjecture is caused by the fact that the balanced metric and the pluriclosed metric in the assumption are allowed to be  different. 

Since its inception, the conjecture has drawn attention from many complex geometers and it has been confirmed in a number of special situations. For example, Verbitsky in \cite{Verbitsky} showed that any twistor space of non-K\"ahlerian type cannot admit any pluriclosed metric, so the Fino-Vezzoni Conjecture holds for all twistor spaces. In \cite{Chiose},  Chiose proved that the conjecture holds for manifolds in the Fujiki ${\mathcal C}$ class, which means compact complex manifolds that are bimeromorphic to compact K\"ahler manifolds. Fu, Li, and Yau in \cite{FuLiYau} proved that the conjecture holds for a special type of non-K\"ahler Calabi-Yau threefolds. Also, Fei \cite{Fei} confirmed the conjecture for generalized Calabi-Gray manifolds, and Otiman \cite{Otiman} proved  that the conjecture  holds for all Oeljeklaus-Toma manifolds. 

In \cite{ZZ-pluriclosed}, Zhao and the second named author proved that the conjecture holds on any compact non-K\"ahler Bismut K\"ahler-like (BKL) manifolds. Here BKL means that the curvature tensor of the Bismut connection obeys all K\"ahler symmetries. Such a metric is always pluriclosed, and when a compact complex manifold admits a non-K\"ahler BKL metric, then it cannot admit any balanced metric.

Next we consider {\em Lie-complex manifolds,} which means compact complex manifolds with universal cover $(G,J)$, where $G$ is a Lie group and $J$ a left-invariant complex structure on $G$. So Lie-complex manifolds are in the form $M=G/\Gamma$ where $\Gamma$ is a discrete subgroup in the automorphism group of $(G,J)$. A commonly occurred special case is when $\Gamma$ is a discrete subgroup of $G$ (acting as left-multiplications on $G$) with smooth compact quotient (i.e., uniform lattice). When $G$ is nilpotent or solvable, $M$ is  called a nilpotent Lie-complex manifold or solvable Lie-complex manifold. In this article, we will simply call it a {\em nilmanifold}  or {\em solvmanifold}. 

Nilpotent groups  admit uniform lattices if and only if they have rational structure constants by the result of Malcev. But for a given solvable group, whether it admits a uniform lattice or not is often a challenging problem in Lie group theory. See \cite{Bock, CM, Garland} as a glimpse of examples. Also, given an even-dimensional Lie group, whether or not it admits a left-invariant complex structure might be complicated as well and it depends on the structure of the Lie group. Clearly, left-invariant complex  structures on $G$ are in one-one correspondence with the {\em complex structures} on its Lie algebra ${\mathfrak g}$, namely, linear maps $J: {\mathfrak g}\rightarrow  {\mathfrak g}$ satisfying $J^2=-I$ and the integrability condition
\begin{equation} \label{integrability}
[x,y] - [Jx,Jy] + J\,[Jx, y] + J\,[x,Jy] = 0, \ \ \ \ \ \forall \ x,y \in {\mathfrak g} 
\end{equation}

By the work of Fino-Grantcharov \cite{FG04}  and Ugarte \cite{Ugarte}, when the complex structure is left-invariant, if the compact quotient $M$ admits a balanced metric, then through the averaging process it also admits a balanced metric that is left-invariant on $G$. Similarly, if $M$ admits a pluriclosed metric, then it also admits a pluriclosed metric whose lift on $G$ is left-invariant. So in order to verify the Fino-Vezzoni Conjecture for a Lie-complex manifold, it suffices to consider {\em Hermitian structures} $(J, g)$ on the Lie algebra ${\mathfrak g}$, where $J$ is a complex structure and $ g = \langle , \rangle $  is an inner product on ${\mathfrak g}$ compatible with $J$ (i.e., making $J$ orthogonal). 

In \cite{FV15}, Fino and Vezzoni proved that their conjecture  holds for all $6$-dimensional nilpotent groups and all $6$-dimensional solvable groups of Calabi-Yau type. In \cite{FV16}, they confirmed the conjecture  for all $2$-step nilpotent groups in any (even) dimension. They also conjectured that if a nilpotent group admits a pluriclosed metric, then it must be of step at most $2$. This last statement was proved recently by Arroyo and Nicolini in \cite{AN}. So in particular we know that Conjecture \ref{conj1} holds for all nilmanifolds.

Beyond nilmanifolds, \cite{FGV} and \cite{Podesta} confirmed the conjecture for all compact semi-simple Lie groups. Giusti and Podest\`a  \cite{GPodesta} proved the conjecture for all regular complex structures on non-compact real semi-simple Lie groups. In their recent work \cite{FP1, FP2, FP3}, Fino and Paradiso confirmed the conjecture in the {\em almost abelian} case (meaning that ${\mathfrak g}$ contains an abelian ideal of codimension $1$) and the {\em almost nilpotent solvable} case (meaning that the nilradical of ${\mathfrak g}$ has $1$-dimensional commutator). In \cite{FSwann22, FSwann}, Freibert and Swann systematically studied the Hermitian structures on $2$-step  solvable Lie algebras, and they proved that Conjecture \ref{conj1} holds for a large subset, namely for those that are `pure type'.

As a natural generalization to the almost abelian case, one could consider those Lie algebras ${\mathfrak g}$ which contain an abelian ideal ${\mathfrak a}$ of codimension $2$. Suppose $J$ is a complex structure on ${\mathfrak g}$. Then ${\mathfrak a}_J= {\mathfrak a} \cap J {\mathfrak a}$ is an ideal in ${\mathfrak g}$ of codimension equal to either $2$ or $4$, depending on whether $J {\mathfrak a} = {\mathfrak a}$ or not. Of course the Hermitian geometric study in the first case (i.e., when ${\mathfrak a}$ is $J$-invariant) is much easier.  

In \cite{GuoZ}, Guo and the second named author studied Hermitian structures on ${\mathfrak g}$ with $J$-invariant ${\mathfrak a}$ and obtained characterizations for balanced, pluriclosed, and other special types of Hermitian metrics. Utilizing this, Li and the second named author in \cite{LiZ} confirmed Conjecture \ref{conj1} in this special case:

\begin{theorem}[Li-Z\,\cite{LiZ}] \label{thm1}
Let ${\mathfrak g}$ be a unimodular Lie algebra which contains an abelian ideal ${\mathfrak a}$ of codimension $2$, and $J$ a complex structure on ${\mathfrak g}$ satisfying $J {\mathfrak a} = {\mathfrak a}$. Then the Fino-Vezzoni Conjecture holds, namely, if $({\mathfrak g} ,J)$ admits compatible metrics $g$ and $h$ so that $g$ is balanced and $h$ is pluriclosed, then it must admit a compatible K\"ahler metric. 
\end{theorem}

Note that ${\mathfrak g}$ is unimodular if $\,\mbox{tr}(\mbox{ad}_x)=0$ for all $x\in {\mathfrak g}$. This is a necessary condition for $G$ to have  compact quotients. So in the study of Fino-Vezzoni Conjecture we always restrict ourselves to unimodular Lie algebras, and without this assumption there are plenty of counterexamples. 

It is  quite natural to wonder what happens when $J {\mathfrak a} \neq {\mathfrak a}$. In this case the Hermitian geometry on $({\mathfrak g},J)$ becomes much more complicated algebraically, and so far we are unable to obtain clear-cut characterizations parallel to  \cite{GuoZ}. However, we are able to get basic descriptions on balanced metrics and pluriclosed metrics, which enable us to confirm the Fino-Vezzoni Conjecture for such $({\mathfrak g} ,J)$. The following is the main result of this article:

\begin{theorem} \label{thm2}
Let ${\mathfrak g}$ be a unimodular Lie algebra which contains an abelian ideal ${\mathfrak a}$ of codimension $2$, and $J$ a complex structure on ${\mathfrak g}$ satisfying $J {\mathfrak a} \neq {\mathfrak a}$. Then the Fino-Vezzoni Conjecture holds. 
\end{theorem}

Combining the two theorems above, we get the following:

\begin{corollary} \label{cor3}
The Fino-Vezzoni Conjecture holds for any unimodular Lie algebra which contains an abelian ideal of codimension $2$. 
\end{corollary}

We would like to point out that, despite the parallelness in the appearance of statement for  the above two theorems, the $J{\mathfrak a}= {\mathfrak a}$ case and the $J{\mathfrak a}\neq {\mathfrak a}$ case are actually quite different. In the former case, if a complex structures $J$ on ${\mathfrak g}$ admits both a balanced metric $g$ and a pluriclosed metric $h$, then through natural algebraic modifications one can  construct  another metric $g_0$ which is K\"ahler. In the latter case, however, for the most part $J$ does not admit any K\"ahler metric at all. So one has to show that $J$ admitting balanced metrics and $J$ admitting pluriclosed metrics will result in conflicting features, so these two types of Hermitian metrics cannot coexist on the same $J$.

The topic of this article lies in the area of non-K\"ahler geometry. The main focus has been  on studying the geometry and topology of various special types of Hermitian structures, where balanced and pluriclosed being two of the most extensively studied. We refer the readers to \cite{AOUV, AT, LW, Salamon, STW, Tosatti, TW,  YZ16, YZZ, ZZ-JGA, Zheng} as a sampler of references.

The article is organized as follows. In \S 2 we will set up notations and collect some basic formula for Lie-Hermitian manifolds. In \S 3 we will consider Lie algebras with codimension $2$ abelian ideals, and analyze properties involving the structure constants. In \S 4, we will turn our attention to balanced metrics on such Lie algebras. In \S 5 we study pluriclosed metrics and establish the proof of the aforementioned Theorem \ref{thm2} in the case when ${\mathfrak g}/{\mathfrak a}$ is non-abelian. Finally, we included two appendices for readers' convenience and to make the presentation more complete. In the first one we give a brief outline of the proof of Theorem \ref{thm1}, as the reference \cite{LiZ} is in Chinese. In the second appendix we prove the `$J{\mathfrak a} \neq {\mathfrak a}$ and ${\mathfrak g}/{\mathfrak a}$ abelian' case of Theorem \ref{thm2}, which was omitted before. Note that in this case ${\mathfrak g}$ must be $2$-step solvable. Hermitian metrics on $2$-step solvable Lie algebras were systematically studied by Freibert and Swann in \cite{FSwann22} and \cite{FSwann}. In particular, Fino-Vezzoni Conjecture was established for such Lie algebras of pure types.

\vspace{0.3cm}

\section{Preliminary on Hermitian Lie algebras}

We start by fixing some terminologies and notations, following \cite{VYZ} and \cite{GuoZ}. Let  $(M^n,g)$ be a  compact Hermitian manifold with universal cover $(G,J,g)$,  where $G$ is an even-dimensional Lie group,  $J$ is a left-invariant complex structure on $G$, and $g$ is a left-invariant Riemannian metric on $G$ compatible with $J$. As we mention before, the compactness of $M$ forces $G$ to be unimodular.

Let ${\mathfrak g}$ be the Lie algebra of $G$, and for convenience we will use the same letter $J$ or $g=\langle , \rangle $ to denote respectively the almost complex structure or inner product on ${\mathfrak g}$ corresponding to that of $G$. The  integrability of $J$ is characterized by the property (\ref{integrability}). 

Denote by ${\mathfrak g}^{\mathbb C}$ the complexification of ${\mathfrak g}$, and by ${\mathfrak g}^{1,0}= \{ x-\sqrt{-1}Jx \mid x \in {\mathfrak g}\} \subseteq {\mathfrak g}^{\mathbb C}$. The condition (\ref{integrability}) simply means that ${\mathfrak g}^{1,0}$ is a complex Lie subalgebra of ${\mathfrak g}^{\mathbb C}$. Let us extend $g=\langle , \rangle $ bilinearly over ${\mathbb C}$, and let $e=\{ e_1, \ldots , e_n\}$ be a unitary basis of ${\mathfrak g}^{1,0}$. Again following \cite{VYZ}, we will use
\begin{equation*} 
C^j_{ik} = \langle [e_i,e_k],\overline{e}_j \rangle, \ \ \ \ \ \  D^j_{ik} = \langle [\overline{e}_j, e_k] , e_i \rangle,
\end{equation*}
to denote the structure constants, or equivalently, under the unitary frame $e$ we have
\begin{equation} \label{CD}
[e_i,e_j] = \sum_k C^k_{ij}e_k, \ \ \ \ \ [e_i, \overline{e}_j] = \sum_k \big( \overline{D^i_{kj}} e_k - D^j_{ki} \overline{e}_k \big) .
\end{equation}
Note that ${\mathfrak g}$ is unimodular if and only if $\mbox{tr}(ad_x)=0$ for any $x\in {\mathfrak g}$, which is equivalent to
\begin{equation} \label{unimodular}
{\mathfrak g} \ \, \mbox{is unimodular}  \ \ \Longleftrightarrow  \ \ \sum_r \big( C^r_{ri} + D^r_{ri}\big) =0 , \, \ \forall \ i.
\end{equation}

Denote by $\nabla$ the Chern connection and by $T$ its torsion tensor. We have
\begin{equation*} \label{Gamma}
\nabla e_i = \sum_j \theta_{ij} e_j, \ \ \ \ \theta_{ij} = \sum_k \big( \Gamma^j_{ik} \varphi_k - \overline{\Gamma^i_{jk}}\, \overline{\varphi}_k \big), \ \ \ \ \ \Gamma^j_{ik} = D^j_{ik},
\end{equation*}
where $\{ \varphi_1, \ldots , \varphi_n\}$ is  the coframe dual to $e$ and $\theta$ is the connection matrix of $\nabla$ under $e$. The torsion tensor $T$ of $\nabla$ has components
\begin{equation} \label{torsion}
T( e_i, \overline{e}_j)=0, \ \ \ \ T(e_i,e_j)  = \sum_k T^k_{ij}e_k, \ \ \ \  \ \ T^j_{ik}= - C^j_{ik} -  D^j_{ik} + D^j_{ki}.
\end{equation}
Denote by $\omega = \sqrt{-1}\sum_i \varphi_i \wedge \overline{\varphi}_i$ the K\"ahler form of the metric $g$.   The structure equation takes the form:
\begin{equation} \label{structure}
d\varphi_i = -\frac{1}{2} \sum_{j,k} C^i_{jk} \,\varphi_j\wedge \varphi_k - \sum_{j,k} \overline{D^j_{ik}} \,\varphi_j \wedge \overline{\varphi}_k.
\end{equation}
Differentiate the above, we get the  first Bianchi identity, which is equivalent to the Jacobi identity in this case:
\begin{equation} \label{Jacobi}
\left\{ 
\begin{split} 
  \sum_r \big( C^r_{ij}C^{\ell}_{rk} + C^r_{jk}C^{\ell}_{ri} + C^r_{ki}C^{\ell}_{rj} \big) \ = \ 0,  \hspace{3.1cm}\\
  \sum_r \big( C^r_{ik}D^{\ell}_{jr} + D^r_{ji}D^{\ell}_{rk} - D^r_{jk}D^{\ell}_{ri} \big) \ = \ 0,  \hspace{2.9cm} \\
  \ \ \sum_r \big( C^r_{ik}\overline{D^r_{j \ell}}  - C^j_{rk}\overline{D^i_{r \ell}} + C^j_{ri}\overline{D^k_{r \ell}} -  D^{\ell}_{ri}\overline{D^k_{j r}} +  D^{\ell}_{rk}\overline{D^i_{jr}}  \big) \ = \ 0 . 
\end{split} 
\right. 
\end{equation}
From the structure equation (\ref{structure}), a direct computation  leads us to the following
\begin{eqnarray*}
\partial (\omega^{n-1}) & = & - \eta \wedge \omega^{n-1}, \ \ \ \ \ \ \eta \ =  \ \sum_i \eta_i \varphi_i, \ \ \ \ \eta_i \ = \ \sum_k T^k_{ki} \, = \, \sum_k D^k_{ki}; \\
\sqrt{-1}\partial \overline{\partial} \omega & = & \sum_{i,j,k,\ell } \sum_r \left( -\frac{1}{2} T^r_{ik}  \overline{C^r_{j\ell }} - T^j_{ir} \overline{D^k_{r\ell }} + T^j_{kr} \overline{D^i_{r\ell}} \right) \varphi_i\varphi_k \overline{\varphi}_j \overline{\varphi}_{\ell}.
\end{eqnarray*}
Note that in the last equality for $\eta_i$, we used the assumption that ${\mathfrak g}$ is unimodular (\ref{unimodular}). In particular, one has $\partial \overline{\partial} \omega^{n-1} = 0$,  namely, {\em any unimodular Hermitian Lie algebra $({\mathfrak g}, J, g)$ is always Gauduchon.} By definition and the proof of Lemma 2 in \cite{GuoZ}, we have the following:

\begin{lemma} \label{lemma1}
Let ${\mathfrak g}$ be a unimodular Lie algebra with a Hermitian structure $(J,g)$.  Then the  metric $g$ is balanced if and only if
\begin{equation} \label{balanced}
\sum_{k} D^k_{ik} = 0 , \ \ \ \ \forall \ 1\leq i\leq n.
\end{equation}
The metric $g$ is pluriclosed if and only if
\begin{equation} \label{SKT}
\sum_r \left( - T^r_{ik}  \overline{C^r_{j\ell }} - T^j_{ir} \overline{D^k_{r\ell }} + T^j_{kr} \overline{D^i_{r\ell}}  +  T^{\ell}_{ir} \overline{D^k_{rj }} - T^{\ell}_{kr} \overline{D^i_{rj}}  \right) = 0,
\end{equation}
for any $1\leq i<k\leq n$ and any $1\leq j< \ell \leq n$.  
\end{lemma}

Another algebraic fact that will be used later is the following:

\begin{lemma} \label{lemma2}
Suppose $A$ and $B$ are two $n\times n$ matrices satisfying $[A,B]=c A$ for some constant $c\neq 0$. Then $A$ must be nilpotent. 
\end{lemma}

Since this is a standard fact in linear algebra and is well-known, we will omit its proof here. 

\vspace{0.3cm}

\section{Structure constants}

Throughout this section, we will assume that {\em  ${\mathfrak g}$ is a unimodular Lie algebra of real dimension $2n$ containing an abelian ideal ${\mathfrak a}$ of codimension $2$,  $J$  a complex structure on ${\mathfrak g}$, and $g=\langle , \rangle $ an inner product on ${\mathfrak g}$ compatible with $J$. }

We will also make a couple of assumptions that will be used throughout the rest of the article:
\begin{equation} \label{assumption}
  J{\mathfrak a} \neq {\mathfrak a}, \ \  \mbox{and} \ \ {\mathfrak g}/{\mathfrak a} \ \mbox{is non-abelian}.
  \end{equation}

\begin{remark}
First of all, note that the $J{\mathfrak a} = {\mathfrak a}$ case has been treated in \cite{GuoZ} and \cite{LiZ}. Since \cite{LiZ} is in Chinese, we will  outline the proof of Fino-Vezzoni Conjecture for this case in Appendix A. Secondly, for the ``$J{\mathfrak a} \neq {\mathfrak a}$ but ${\mathfrak g}/{\mathfrak a}$ abelian" case, the Lie algebra ${\mathfrak g}$ must be $2$-step solvable. The Hermitian structures on $2$-step solvable Lie algebras were treated systematically by Freibert and Swann in \cite{FSwann22} and \cite{FSwann} using powerful algebraic methods. In particular they established Fino-Vezzoni Conjecture for all $2$-step solvable Lie algebras of pure types. Since our ${\mathfrak g}$ here might not be of pure types, we included the proof of Theorem \ref{thm2} for this case in Appendix B, which is somewhat lengthy but methodwise analogous to the ${\mathfrak g}/{\mathfrak a}$ non-abelian case.  
\end{remark}
Write
\begin{equation} \label{a'andb}
  {\mathfrak a}_J:= {\mathfrak a} \cap J {\mathfrak a}, \ \ \ \ \ {\mathfrak a}':=  {\mathfrak a} + [{\mathfrak g}, {\mathfrak g}], \ \ \ \ \ {\mathfrak b} = {\mathfrak a} \cap J {\mathfrak a}'.
\end{equation}
Under our assumption (\ref{assumption}),  ${\mathfrak a}_J$  and ${\mathfrak a}'$ are  of codimension $4$ and $1$ in ${\mathfrak g}$, respectively. Using the integrality condition (\ref{integrability}), it is not hard to show that

\begin{lemma} \label{lemma3}
Let $({\mathfrak g},J,g)$ be a unimodular Hermitian Lie algebra which contains an abelian ideal of codimension $2$ satisfying (\ref{assumption}). Then ${\mathfrak a}_J$ is an ideal of ${\mathfrak g}$ and 
$${\mathfrak a}_J \subsetneq {\mathfrak b}  \subsetneq {\mathfrak a} \subsetneq {\mathfrak a}' \subsetneq {\mathfrak g} .$$
Furthermore, if $x\in {\mathfrak b}\setminus {\mathfrak a}_J$, then $Jx \in {\mathfrak a}' \setminus {\mathfrak a}$; while if $y\in {\mathfrak a}\setminus {\mathfrak b}$, then $Jy \notin {\mathfrak a}' $.
\end{lemma}

\begin{definition} [{\bf admissible frame}]
Let $({\mathfrak g},J,g)$ be a unimodular Hermitian Lie algebra which contains an abelian ideal ${\mathfrak a}$ of codimension $2$ satisfying (\ref{assumption}). Suppose $\dim_{ \mathbb R}{\mathfrak g}=2n$. A unitary basis $\{ e_1, \ldots , e_n\}$ of ${\mathfrak g}^{1,0}$ is called an {\bf admissible frame} of ${\mathfrak g}$, if 
\begin{equation*} \label{def}
{\mathfrak a}_J = \mbox{span}\{ e_j+\overline{e}_j, i(e_j-\overline{e}_j); \ 3\leq j\leq n\}, \ \ \ \ {\mathfrak a} ={\mathfrak a}_J + \mbox{span}\{x,y\}, 
\end{equation*} 
where $\{ x,y\}$ is perpendicular to ${\mathfrak a}_J$ and orthonormal, with $x\in {\mathfrak b}$ and 
\begin{equation*} \label{def2}
e_1=\frac{1}{\sqrt{2}}(x-iJx), \ \ \ Y= \frac{1}{\sqrt{2}}(y-iJy) = i\delta e_1 + \delta'e_2.
\end{equation*} 
Here $\delta = \langle Jx, y\rangle \in (-1,1)$ and $\delta '=\sqrt{1-\delta^2} \in (0,1]$. 
\end{definition}

Note that once the metric $g$ is given on $({\mathfrak g}, J)$, both $x$ and $y$ are uniquely determined up to $\pm $ sign, as $x\in {\mathfrak b}\cap  {\mathfrak a}_J^{\perp}$ and $y\in {\mathfrak a}\cap  {\mathfrak b}^{\perp}$.

\begin{lemma} \label{lemma4}
Let $({\mathfrak g},J,g)$ be a unimodular Hermitian Lie algebra which contains an abelian ideal ${\mathfrak a}$ of codimension $2$ satisfying (\ref{assumption}), and $e$ an admissible frame. Then the structure constants $C$ and $D$ satisfy 
\begin{eqnarray*}
&& C^{\ast}_{ij} = D^{\ast}_{\ast i} = C^{\alpha }_{\beta i} = D^i_{\alpha\beta} =0,   \label{eqCD1}\\
&& C^{\ast}_{1i}= \overline{D^i_{\ast 1}} , \ \ \ C^{\ast}_{2i}= \overline{D^i_{\ast 2}} -2t \overline{D^i_{\ast 1}}, \ \ \ C^{\ast}_{12}= \overline{D^2_{\ast 1}} -  \overline{D^1_{\ast 2}} + 2t \overline{D^1_{\ast 1}},  \label{eqCD2}
\end{eqnarray*}
for any $1\leq \alpha , \beta \leq 2$ and any $3\leq i,j\leq n$, any $1\leq \ast \leq n$. Here $t=\frac{i\delta}{\delta'}$.
\end{lemma}

\begin{proof}
These properties are direct consequences of the fact that ${\mathfrak a}$ is abelian and is an ideal. For instance, from the fact that $[{\mathfrak a}_J, {\mathfrak a}_J]=0$ we get $[e_i,e_j]=[e_i, \overline{e}_j]=0$ for any $3\leq i,j\leq n$, hence by (\ref{CD}) we have $C^{\ast }_{ij}=D^i_{\ast j}=0$. The other properties follows similarly, and  we  omit the straight forward calculations here.
\end{proof}

Based on the above lemma, under an admissible frame $e$ all the  $C$ components can be expressed in terms of $D$ components, and the only possibly non-zero components of $D$ are the following:
\begin{equation} \label{11}
D_{\alpha} = (D^{\,\ast}_{\ast \alpha}) = \left[ \begin{array}{ll} E_{\alpha} & 0 \\ V_{\alpha} & Y_{\alpha}  \end{array} \right] , \ \ \ \ \ E_{\alpha} = \left[ \begin{array}{ll} D^{\,1}_{1\alpha} &  D^{\,2}_{1\alpha} \\  D^{\,1}_{2\alpha} &  D^{\,2}_{2\alpha}  \end{array} \right] , \ \ \ \ \ V_{\alpha} = [v^1_{\alpha} , v^2_{\alpha}] .\ \ \ \ 
\end{equation}
for $\alpha =1$ and $2$, where each $v^{\beta}_{\alpha}$ (for $\beta =1,2$) is a column vector in ${\mathbb C}^{n-2}$, while each $Y_{\alpha} =(D^{j}_{i\alpha})$ is an $(n-2)\times (n-2)$ matrix. In order to prove the Fino-Vezzoni Conjecture, we need to  find the exact value of the entries of $E_1$ and $E_2$. Towards this end, let us write

\begin{equation} \label{abcd}
\left\{ \begin{split} 
[Jx, x] =  ax + by + {\mathfrak a}_J ;\ \,\\ 
[Jx, y] =  cx + dy + {\mathfrak a}_J ;\ \,\\ 
\ [Jy, x] =  a'x + b'y + {\mathfrak a}_J ;\\
\ [Jy, y] =  c'x + d'y + {\mathfrak a}_J.
\end{split}
\right.
\end{equation}
where $a, \ldots , d'$ are real constants. By the choice of admissible frames, $x\in {\mathfrak b}={\mathfrak a} \cap J {\mathfrak a}' $,  so $Jx \in J {\mathfrak b} \subseteq {\mathfrak a}'$. Thus  $[Jx,Jy]=\sigma Jx + {\mathfrak a}$ for some constant $\sigma$, and this constant cannot be zero as otherwise ${\mathfrak g}/{\mathfrak a}$ would be abelian, contradicting with  (\ref{assumption}). By the integrability condition (\ref{integrability}) and the fact that $[x,y]=0$, we get
\begin{equation*} \label{JxJy}
[Jx, Jy] =  J \{ [Jx,y] - [Jy, x] \} = (c-a')Jx + (d-b')Jy + {\mathfrak a}_J = \sigma Jx + {\mathfrak a}.
\end{equation*}
Therefore we must have
\begin{equation} \label{JxJy}
[Jx, Jy] =   \sigma Jx + {\mathfrak a}_J \ \ \ \ \mbox{and} \ \ \ \ a'=c-\sigma , \ \ b'=d.
\end{equation}
On the other hand, by Jacobi identity, the equation
$$ [[Jx, Jy], z] = [ Jx, [Jy, z]] - [Jy , [Jx, z]] $$
holds for any $z \in {\mathfrak g}$. Applying this to $z=x$ and $z=y$, we get
$$ [ A, B] = - \sigma A, \ \ \ \mbox{where} \ \ A= \left[ \begin{array}{ll} a & b \\ c & d \end{array} \right] , \ \ B= \left[ \begin{array}{ll} a' & b' \\ c' & d' \end{array} \right]  .$$

By Lemma \ref{lemma2}, we know that the  $2\times 2$ matrix $A$ must be nilpotent, so it has vanishing trace and determinant. Therefore we have
\begin{equation} \label{abcd1}
d=-a, \ \ \ a^2+bc = 0, \ \ \ a'=c-\sigma , \ \ \ b'= -a, 
\end{equation}
while the equation $[A,B]=-\sigma A$ becomes 
\begin{equation} \label{abcd2}
\left\{ \begin{split} \,bc' = -a (c+\sigma ) ; \\
\ bd' = a^2 - 2b \sigma ;\ \,\\
2ac' + c d' = c^2 .\ 
\end{split} \right.
\end{equation}
For convenience in future discussions, let us introduce the following terminology:

\begin{definition} [{\bf types of $J$}]
Let ${\mathfrak g}$ be a unimodular Lie algebra containing a codimension $2$ abelian ideal ${\mathfrak a}$ such that ${\mathfrak g}/{\mathfrak a}$ is not abelian, and $J$ a complex structure on ${\mathfrak g}$ such that $J{\mathfrak a} \neq {\mathfrak a}$. We will call $({\mathfrak g},J)$ to be  {\bf generic} if $[J{\mathfrak b}, {\mathfrak b}] \not\subseteq {\mathfrak b}$;  call it {\bf half-generic} if  $[J{\mathfrak b}, {\mathfrak b}] \subseteq {\mathfrak b}$ and  $[J{\mathfrak b}, {\mathfrak a}] \not\subseteq {\mathfrak a}_J$; and call it {\bf degenerate} if    $[J{\mathfrak b}, {\mathfrak a}] \subseteq {\mathfrak a}_J$.
\end{definition}

Clearly, these three cases are mutually disjoint, and any $J$ with $J{\mathfrak a} \neq {\mathfrak a}$ has to be one of them. Also,  for any given metric $g$ compatible with $J$, then under any admissible frame we have
$$ \left\{ \begin{split} J \ \mbox{is generic} \, \Longleftrightarrow \, b\neq 0; \hspace{2.2cm} \\
J \ \mbox{is half-generic} \, \Longleftrightarrow \, b=0 \ \mbox{and} \ c\neq 0; \\
J \ \mbox{is degenerate} \, \Longleftrightarrow \, b=c= 0. \hspace{1.1cm}
\end{split} \right.  $$

\begin{lemma} \label{lemma5}
Let $({\mathfrak g},J,g)$ be a unimodular Hermitian Lie algebra which contains an abelian ideal ${\mathfrak a}$ of codimension $2$ satisfying (\ref{assumption}), and $e$ an admissible frame. Then we have 
\begin{equation*} \label{abcd3}
\left\{ \begin{split} \mbox{When} \ J \ \mbox{is generic}: \  \  \ b\neq 0, \ c=-\frac{a^2}{b}, \ c'=-\frac{a}{b}(c+\sigma ), \ d'= -c -2\sigma ;  \ \\
\mbox{When} \ J \ \mbox{is half-generic}: \ \   c\neq 0, \ a=b=0, \ d'=c, \ \mbox{and} \ c' \ \mbox{is artbitrary}; \\
\mbox{When} \ J \ \mbox{is degenerate}:  \ \ \  a=b=c= 0, \ c' \ \mbox{and} \ \, d' \ \mbox{are artbitrary.} \hspace{1.2cm}
\end{split} \right.  
\end{equation*}
\end{lemma}

Note that we always have $\sigma \neq 0$ and $a'=c-\sigma$, $d=b'=-a$ by (\ref{JxJy}). In each of the above three cases, the data is determined by three independent real constants: by $b\sigma \neq 0$ and $a$ in the generic case; by $c\sigma \neq 0$ and $c'$ in the half-generic case; by  $\sigma \neq 0$ and $c'$, $d'$ in the degenerate case. We have the following:

\begin{lemma} \label{lemma6}
Let $({\mathfrak g},J,g)$ be a unimodular Hermitian Lie algebra which contains an abelian ideal ${\mathfrak a}$ of codimension $2$ satisfying (\ref{assumption}), and $e$ an admissible frame. Then we have 
\begin{equation} \label{C12}
C^1_{12} = -\frac{\ i\sigma }{\sqrt{2}\delta'}, \ \ \ C^2_{12}=0, 
\end{equation}
where $\sigma \neq 0$ is the constant given by $[Jx,Jy]=\sigma Jx + {\mathfrak a}_J$ in (\ref{JxJy}). 
\end{lemma}

\begin{proof}
By the choice of admissible frames, we have
\begin{eqnarray*}
2 \delta' [e_1, e_2] & = & 2 [e_1, Y] \ = \ [x-iJx, y-iJy] \\
& = & - [Jx, Jy] + i \{ [Jy,x] - [Jx, y] \} \\
& \equiv & - \sigma Jx + i (-\sigma x) \\
& = & -i \sigma  (x - iJx) \ = \ -i \sigma  \sqrt{2} e_1.
\end{eqnarray*}
Here and from now on `$\equiv$' means equality modular ${\mathfrak a}_J \otimes {\mathbb C}$. On the other hand, by (\ref{CD}), we know that $[e_1, e_2]\equiv C^1_{12}e_1 + C^2_{12}e_2$, hence the above computation gives us the expression of $C^1_{12}$ and $C^2_{12}$, and the proof of Lemma \ref{lemma6} is completed.
\end{proof}

Following similar pattern of computation, we have
\begin{eqnarray*}
 [e_1, \overline{e}_1] & = & -i [Jx, x] \ \equiv  \ -iax-iby \\
& = & - \frac{\,ia}{\sqrt{2}} (e_1+ \overline{e}_1)  -  \frac{\,ib}{\sqrt{2}} (Y+ \overline{Y})   \\
& = & - \frac{\,ia}{\sqrt{2}} (e_1+ \overline{e}_1)  -  \frac{\,ib}{\sqrt{2}} (i\delta e_1 + \delta' e_2 - i\delta \overline{e}_1 + \delta' \overline{e}_2)   \\ 
& = & \frac{b\delta -ia}{\sqrt{2}} e_1   - \frac{\,ib\delta'}{\sqrt{2}} e_2    -    \frac{b\delta +ia}{\sqrt{2}}   \overline{e}_1 -  \frac{\,ib\delta'}{\sqrt{2}} \overline{e}_2.
\end{eqnarray*}
Compare this with formula (\ref{CD}) for $[e_1, \overline{e}_1] $, we obtain the following
\begin{equation} \label{D1}
D^1_{11} = \frac{b\delta +ia}{\sqrt{2}} , \ \ \ \ D^1_{21} = \frac{\,ib\delta'}{\sqrt{2}} .
\end{equation}
Analogously, we compute
\begin{eqnarray}
 2 [Y, \overline{e}_1] & = & [y-iJy, x+iJx] \ = \  - [Jx,Jy] - i \{ [Jx, y] + [Jy, x] \} \nonumber \\
& \equiv  & -\sigma Jx -i (2c-\sigma)x +i2ay  \nonumber \\
& = & i\sigma (x+iJx) -i2c x + i2a y  \nonumber \\ 
& = & i\sigma \sqrt{2}\, \overline{e}_1 -ic \sqrt{2} (e_1 + \overline{e}_1) +ia \sqrt{2} (Y + \overline{Y}) \nonumber \\
& = & i\sigma \sqrt{2}\, \overline{e}_1 -ic \sqrt{2} (e_1 + \overline{e}_1) +ia \sqrt{2} (i\delta e_1 + \delta' e_2 - i\delta \overline{e}_1 + \delta'  \overline{e}_2)  \nonumber \\
& = & -(ic+ a \delta ) \sqrt{2} \,e_1   + ia\delta' \sqrt{2}\, e_2    + \{ a\delta + i(\sigma -c) \}   \sqrt{2}   \, \overline{e}_1 + ia \delta' \sqrt{2}\,  \overline{e}_2. \label{Ye1}
\end{eqnarray}
On the other hand, we have
\begin{eqnarray*}
 \sqrt{2}\, [Y, \overline{e}_1] & = & i\delta \sqrt{2}\, [e_1,\overline{e}_1] + \delta' \sqrt{2}\, [e_2,\overline{e}_1]   \nonumber \\
& \equiv  & i\delta \{ (b\delta -ia) e_1   - ib\delta' e_2    -    (b\delta +ia)  \overline{e}_1 -  ib\delta' \overline{e}_2 \} + \delta' \sqrt{2}\, [e_2,\overline{e}_1]   \nonumber \\
& = & \{ (a\delta + ib\delta^2) e_1   + b\delta \delta' e_2    +    (a\delta - ib\delta^2)  \overline{e}_1 + b\delta \delta' \overline{e}_2 \} + \delta' \sqrt{2}\, [e_2,\overline{e}_1]   .
\end{eqnarray*}
Combining this with (\ref{Ye1}), we obtain
\begin{equation*}
  \delta'\sqrt{2}\, [e_2, \overline{e}_1] \ \equiv  \  -(2a\delta +ic+ib\delta^2) e_1   + (- b\delta +ia)\delta'  e_2  +i(\sigma -c +b\delta^2)  \overline{e}_1    +(ia-b\delta )\delta' \overline{e}_2 . 
\end{equation*}
Compare it with formula (\ref{CD}) for $[e_2, \overline{e}_1]$, which is
\begin{equation*}
  [e_2, \overline{e}_1] \ \equiv  \  \overline{D^2_{11} } e_1   + \overline{D^2_{21} } e_2  - D^1_{12} \overline{e}_1    -  D^1_{22} \overline{e}_2 , 
\end{equation*}
we derive at the following
\begin{eqnarray}
 && D^2_{11} \ = \ \frac{1}{\sqrt{2}\delta'} (-2a\delta +ic+ib\delta^2), \ \ \ D^2_{21} \ = \  - \frac{1}{\sqrt{2}} (b\delta + ia ) , \label{D2}\\
&& D^1_{12} \ = \ \frac{1}{\sqrt{2}\delta'} i(c-\sigma -b\delta^2),  \ \ \ \ \ \ D^1_{22} \ = \ \frac{1}{\sqrt{2}}(b\delta -ia).  \label{D3}
\end{eqnarray}
Analogously, consider
\begin{eqnarray*}
 [Y, \overline{Y}] & = & -i [Jy, y] \ \equiv \ -ic'x -id' y   \nonumber \\
& =  & -\frac{ic'}{\sqrt{2}} (e_1 + \overline{e}_1) -\frac{id'}{\sqrt{2}} \big( i\delta e_1 +\delta'e_2 - i\delta \overline{e}_1 + \delta' \overline{e}_2\big)  \nonumber \\
& = & \frac{d'\delta -ic'}{\sqrt{2}} e_1 - \frac{id'\delta'}{\sqrt{2}} e_2 - \frac{d'\delta +ic'}{\sqrt{2}} \overline{e}_1  - \frac{id'\delta'}{\sqrt{2}} \overline{e}_2  , 
\end{eqnarray*}
On the other hand,
\begin{eqnarray*}
  [Y, \overline{Y}] & = & -i\delta [Y, \overline{e}_1] + \delta' [Y, \overline{e}_2] \ = \  -i\delta [Y, \overline{e}_1] +i\delta \delta' [e_1, \overline{e}_2]  + \delta'^2 [e_2, \overline{e}_2]   \nonumber \\
& =  &  -i\delta [Y, \overline{e}_1] +   i\delta  [e_1, \overline{Y - i\delta e_1} ]   + \delta'^2 [e_2, \overline{e}_2]  \nonumber \\
& =  &  -i\delta \{ [Y, \overline{e}_1] +  [\overline{Y},e_1] \} - \delta^2  [e_1, \overline{e}_1]    + \delta'^2 [e_2, \overline{e}_2] . 
\end{eqnarray*}
So by the expressions for $[Y, \overline{e}_1]$, $[e_1, \overline{e}_1]$, and $[Y, \overline{Y}]$, we get
\begin{eqnarray*}
  \delta'^2 [e_2, \overline{e}_2] & = & [Y, \overline{Y}]   + \delta^2 [e_1, \overline{e}_1 ] + i\delta \{ [Y, \overline{e}_1] +  [\overline{Y},e_1] \}  \nonumber \\
& =  &  [Y, \overline{Y}]   + \delta^2 [e_1, \overline{e}_1 ]   + \frac{\sigma \delta }{\sqrt{2}} e_1 - \frac{\sigma \delta }{\sqrt{2}} \overline{e}_1 \nonumber \\
& =  &   [Y, \overline{Y}]   +  \frac{(b\delta^2+\sigma) \delta -ia\delta^2 }{\sqrt{2}} e_1-   \frac{(b\delta^2+\sigma) \delta + ia\delta^2 }{\sqrt{2}} \overline{e}_1  - \frac{ib\delta^2\delta' }{\sqrt{2}} (e_2+\overline{e}_2) \nonumber \\
& =  & \frac{1}{\sqrt{2}} \big( -i (c' + a \delta^2) + \delta (d'+b\delta^2 +\sigma ) \big) e_1 - \frac{1}{\sqrt{2}} i(d'+b\delta^2)\delta' e_2 \nonumber \\
& & - \frac{1}{\sqrt{2}} \big( i (c'+a \delta^2) + \delta (d'+b\delta^2 + \sigma ) \big)  \overline{e}_1 - \frac{1}{\sqrt{2}} i(d'+b\delta^2)\delta' \overline{e}_1 .
\end{eqnarray*}
Compare this with formula (\ref{CD}) for $[e_2, \overline{e}_2] $, which is
$$ [e_2, \overline{e}_2] \ \equiv \ \overline{D^2_{12}} e_1 + \overline{D^2_{22}} e_2 - D^2_{12} \overline{e}_1 - D^2_{22} \overline{e}_2, $$
we obtain the last two entries of $E_2$:
\begin{equation} \label{D4}
D^2_{12} \ = \ \frac{1}{\sqrt{2}\delta'^2} \big( i(c'+a\delta^2) + \delta (d'+b\delta^2+ \sigma ) \big) , \ \ \ \ D^2_{22} \ = \   \frac{1}{\sqrt{2}\delta'} i(d'+ b\delta^2 ) .
\end{equation}
Now (\ref{C12}), (\ref{D1}), (\ref{D2}), (\ref{D3}), (\ref{D4}) give us the key ingredients for our later discussions. Since $C$ can be expressed in $D$, and $D^{\ast}_{\ast i}=0$ for all $i\geq 3$, the Jacobi identities (\ref{Jacobi}) become 
\begin{equation*} \label{Jacobi2}
[D_1, D_2] = - C^1_{12} D_1 = \frac{i\sigma}{\sqrt{2}\delta'} D_1, 
\end{equation*}
or equivalently,
\begin{eqnarray}
  && [E_1, E_2] \ = \ \frac{i\sigma}{\sqrt{2}\delta'} E_1, \ \ \ \ \  \ \ \ [Y_1, Y_2] \ = \ \frac{i\sigma}{\sqrt{2}\delta'} Y_1,   \label{Jacobi1}\\
  && V_1E_2+Y_1V_2-V_2E_1 - Y_2V_1 \ = \ \frac{i\sigma}{\sqrt{2}\delta'} V_1. \label{Jacobi2}
\end{eqnarray}
In particular, by Lemma \ref{lemma2} we know that both $E_1$ and $Y_1$ must be nilpotent matrices.

\vspace{0.3cm}

\section{Balanced metrics}

Again let $({\mathfrak g}, J, g)$ be a Hermitian Lie algebra which contains an abelian ideal ${\mathfrak a}$ of codimension $2$ such that (\ref{assumption}) is satisfied. In this section we will analyze the conditions for the metric $g$ to be balanced. 

Recall that the Lie algebra ${\mathfrak g}$ is said to be {\em unimodular} if $\mbox{tr}(\mbox{ad}_x)=0$ holds for all $x\in {\mathfrak g}$. This is independent of the choice of complex structures or metrics, and is a necessary condition for the corresponding Lie group $G$ to have any compact quotient. When $(J, g)$ is a Hermitian structure on ${\mathfrak g}$, then  ${\mathfrak g}$ is unimodular when and only when the equality in (\ref{unimodular}) is satisfied. In our case, the condition is equivalent by Lemma \ref{lemma4} to the following:
$$ \mbox{tr}(E_1) + \mbox{tr}(Y_1)  - \overline{\mbox{tr}(Y_1) } =0, \ \ \ \ C^1_{12} + \mbox{tr}(E_2) + \mbox{tr}(Y_2)  - \overline{\mbox{tr}(Y_2) } +2t \, \overline{\mbox{tr}(Y_1) } =0. $$
Since $E_1$ and $Y_1$ are both nilpotent, their trace vanish, so we conclude the following

\begin{lemma} \label{lemma7}
Let $({\mathfrak g}, J, g)$ be a Hermitian Lie algebra containing an abelian ideal ${\mathfrak a}$ of codimension $2$ satisfying (\ref{assumption}). Then ${\mathfrak g}$ is unimodular if and only if 
\begin{equation*}
\mbox{tr}(Y_2)  - \overline{\mbox{tr}(Y_2) } = \frac{i}{\sqrt{2}\delta' } \big( 2 \sigma - d' -c \big). 
\end{equation*}
\end{lemma}

Next we want to find out when will the metric $g$ be balanced. By Lemma \ref{lemma1} and Lemma \ref{lemma4}, we know that this happens when 
$ \sum_s D^s_{\ast s} =  D^1_{\ast 1} + D^2_{\ast 2} = 0$ holds
for any $\ast$. For all  $\ast \geq 3$, this means that $v^1_1 + v^2_2=0$, while for $\ast =1$ and $2$, it means that
\begin{eqnarray*}
&& 0 \ = \  D^1_{11}+D^2_{12} \ = \ \frac{1}{\sqrt{2}\delta'^2} \big( i(c'+a) + \delta (d'+b+\sigma )\big) ,\\
&& 0 \ = \  D^1_{21}+D^2_{22} \ = \ \frac{1}{\sqrt{2}\delta'} i (d'+b).
\end{eqnarray*}
Here we used formula (\ref{D1}), (\ref{D4}), and the fact that $\delta^2+\delta'^2=1$. Since $\sigma \neq 0$, the above means 
\begin{equation*} \label{balanced1}
c'=-a, \ \ \ d'=-b, \ \ \ \delta =0.
\end{equation*}
Combining this with (\ref{abcd1}) and (\ref{abcd2}), we see that $a=c=0$ and $b(b-2\sigma )=0$. So either $b=0$, and we are in the $J$ degenerate case, or $b=2\sigma$ so we are in the $J$ generic case. In the latter case $b$ and $\sigma$ happen to have the same sign. In summary, we have proved the following:

\begin{lemma} \label{lemma8}
Let ${\mathfrak g}$ be a unimodular Lie algebra containing an abelian ideal ${\mathfrak a}$ of codimension $2$, and $J$ a complex structure on ${\mathfrak g}$. Assume that (\ref{assumption}) holds. If $({\mathfrak g}, J)$ admits a balanced Hermitian metric $g$, then either $J$ is degenerate with $d'=0$, or $J$ is generic and with $b$ and $\sigma $ having the same sign.
\end{lemma} 

We recall that $J$ is degenerate when and only when $b=c=0$, and $J$ is generic when and only when $b\neq 0$. We claim that in the $J$ generic case, the sign of the non-zero constants $\sigma b$ depends only on $J$, while when $J$ is degenerate, $d'=0$ is a condition independent of the choice of metrics. In other words, we have the  following two lemmas: 

\begin{lemma} \label{lemma9}
Let ${\mathfrak g}$ be a Lie algebra containing an abelian ideal ${\mathfrak a}$ of codimension $2$, and $J$ a complex structure on ${\mathfrak g}$. Assume that (\ref{assumption}) holds. Suppose $J$  is generic and $g$ is any Hermitian metric on $({\mathfrak g}, J)$, then the sign of  $\sigma b $ depends only on $J$, and is independent of the choice of $g$. 
\end{lemma}

\begin{proof}
Under the metric $g$, let us take an orthonormal pair $\{ x, y\}$ in ${\mathfrak a}\cap {\mathfrak a}_J^{\perp}$ so that $x\in {\mathfrak b}$. Then the equations (modular the ideal ${\mathfrak a}_J$)
$$ [Jx, Jy] = \sigma Jx + {\mathfrak a}_J, \ \ \ \ \ \ [Jx, x] = ax + by +  {\mathfrak a}_J ,$$
give us the constants $\sigma$ and $b$. Now suppose $\tilde{g}$ is another Hermitian metric on $({\mathfrak g}, J)$, and we take an $\tilde{g}$-orthonormal pair $\{ \tilde{x},\tilde{y} \}$ in ${\mathfrak a}$ so that $\tilde{x}\in {\mathfrak b}$ and both $\tilde{x}$ and $\tilde{y}$ are perpendicular to ${\mathfrak a}_J$ with respect to $\tilde{g}$. We have
$$ [J\tilde{x}, J\tilde{y}] = \tilde{\sigma} J\tilde{x} + {\mathfrak a}_J, \ \ \ \ \ \ [J\tilde{x}, \tilde{x}] = \tilde{a}\tilde{x} + \tilde{b}\tilde{y} +  {\mathfrak a}_J .$$
Since both $x$ and $\tilde{x}$ are in ${\mathfrak b}$, we have
$$ \tilde{x}=\lambda_1 x+u, \ \ \ \tilde{y} = \mu x + \lambda_2 y + v,  $$
for some constants $\lambda_1$, $\lambda_2$, $\mu$ with $\lambda_1\lambda_2\neq 0$ and some $u$, $v\in {\mathfrak a}_J$. Plug into the above defining equations for $\sigma$ and $b$, we get
$$ \tilde{\sigma } = \lambda_2 \sigma, \ \ \ \ \tilde{b} = \frac{\lambda_1^2}{\lambda_2} b.   $$
From this we get $\tilde{\sigma}\tilde{b}=\lambda_1^2 \sigma b$. This shows that for all Hermitian metrics on $({\mathfrak g}, J)$, the product $\sigma b$ have the same sign, which is either positive or negative, depending only on $J$ and not on the metric. 
\end{proof}

\begin{lemma} \label{lemma10}
Let ${\mathfrak g}$ be a Lie algebra containing an abelian ideal ${\mathfrak a}$ of codimension $2$, and $J$ a complex structure on ${\mathfrak g}$. Assume that (\ref{assumption}) holds. Suppose $J$  is degenerate and $g$ is a Hermitian metric on $({\mathfrak g}, J)$ satisfying $d'=0$, then for any other Hermitian metric $\tilde{g}$ on $({\mathfrak g}, J)$, the corresponding constant $\tilde{d'}=0$ as well. 
\end{lemma}

\begin{proof} By definition, in the degenerate case we have
$$ [Jx, x] \equiv [Jx, y] \equiv 0, \ \ \ [Jy,x] \equiv -\sigma x, \ \ \ [Jy, y] \equiv c'x+d'y, $$
where `$\equiv$' means equality modular ${\mathfrak a}_J$. Therefore $d'=0$ when and only when  $[{\mathfrak g} , {\mathfrak a}] \subseteq {\mathfrak b}$. This condition is clearly independent of the choice of the metric. This completes the proof of the lemma. 
\end{proof}

\vspace{0.3cm}

\section{Pluriclosed metrics}

Let $({\mathfrak g}, J, g)$ be a Hermitian Lie algebra which contains an abelian ideal ${\mathfrak a}$ of codimension $2$ such that (\ref{assumption}) is satisfied. In this section we will analyze the conditions for the metric $g$ to be pluriclosed. By Lemma \ref{lemma1}, the metric $g$ is pluriclosed if and only if equation (\ref{SKT}) is satisfied. Let $k=\alpha$ and $\ell = \beta$ be in $\{ 1,2\}$, $i$ and $j$ be in $\{ 3, \ldots , n\}$, then by Lemma \ref{lemma4}, equation (\ref{SKT}) becomes
$$ \sum_{r=1}^n \{ -T^r_{i\alpha} \overline{C^r_{j\beta} } -  T^j_{ir} \overline{D^{\alpha}_{r\beta} }  +   T^j_{\alpha r} \overline{D^{i}_{r\beta} } \} = 0. $$
Since $C^{\sigma}_{j\beta}=D^i_{\sigma \beta}=0$ for any $\sigma =1,2$, we know that for the first and third terms on the left hand side of the above equation, $r$ sums from $3$ to $n$. Similarly, by $T^a_{bc}= -C^a_{bc} - D^a_{bc} + D^a_{cb}$, we have $T^{\ast}_{ij}=0$ for any $i,j\geq 3$, thus for the second term in the above equation $r$ runs from $1$ to $2$. Therefore, the above equation becomes:
$$ \sum_{r=3}^n \{ ( C^r_{i\alpha} + D^r_{i\alpha}) \overline{C^r_{j\beta} }  +  ( -C^j_{\alpha r} + D^j_{r\alpha } ) \overline{D^{i}_{r\beta} } \} + \sum_{\sigma =1}^2  (C^j_{i\sigma }+  D^j_{i\sigma }) \overline{ D^{\alpha}_{\sigma \beta} }  =0.     $$ 
Again by Lemma \ref{lemma4}, we know that
$$ C^{\ast}_{\alpha i} = \overline{ D^i_{\ast \alpha } } -2t \delta_{\alpha 2} \overline{ D^i_{\ast 1} } . $$
Plug into the previous line we get
\begin{eqnarray*}
&& \sum_{r=3}^n \{  \big( - \overline{ D^i_{r \alpha } } + 2t \delta_{\alpha 2} \overline{ D^i_{r 1 } } + D^r_{i\alpha} \big)  \big(- D^j_{r\beta} + 2\bar{t} \delta_{\beta 2} D^j_{r1} \big) + 
\big( - \overline{ D^r_{j \alpha } } + 2t \delta_{\alpha 2} \overline{ D^r_{j 1 } } + D^j_{r\alpha } \big) \overline{D^{i}_{r\beta} } \, \} + \\
&& + \ \,\big(- \overline{ D^i_{j 1 } }  + D^j_{i1} \big) \,\overline{ D^{\alpha}_{1\beta} } + \big(- \overline{ D^i_{j 2 } }  + 2t \overline{ D^i_{j 1 } } + D^j_{i2}  \big) \,\overline{ D^{\alpha}_{2\beta} } \  \ = \ \ 0.
\end{eqnarray*}
Since $\bar{t}=-t$, the matrix form of the above equation becomes
\begin{eqnarray}
&& \big( Y_{\alpha} - Y_{\alpha}^{\ast} + 2t \delta_{\alpha 2} Y_1^{\ast}  \big)  \big(- Y_{\beta} -2t  \delta_{\beta 2} Y_1 \big) + 
 Y^{\ast}_{\beta} \big( Y_{\alpha} - Y_{\alpha}^{\ast}  + 2t \delta_{\alpha 2} Y_1^{\ast} \big)  +   \nonumber \\
&& + \ \, \overline{ D^{\alpha}_{1\beta} }  \, \big( Y_1-Y_1^{\ast}\big)   + \overline{ D^{\alpha}_{2\beta} }\, \big(  Y_2-Y_2^{\ast}  +2t Y_1^{\ast}  \big) \, \  \ = \ \ 0,  \label{SKT2}
\end{eqnarray}
for any $1\leq \alpha , \beta \leq 2$. Here $Y_{\alpha}=(D^j_{i\alpha})$ is the $(n-2)\times (n-2)$ matrix which is the lower right corner of $D_{\alpha}$, $t=\frac{i\delta}{\delta'}$, and $A^{\ast}$ stands for the conjugate transpose of $A$. The above equation is a necessary condition for the metric $g$ to be pluriclosed.

\begin{lemma} \label{lemma11}
Let ${\mathfrak g}$ be a unimodular Lie algebra containing an abelian ideal ${\mathfrak a}$ of codimension $2$, and $J$ a complex structure on ${\mathfrak g}$. Assume that (\ref{assumption}) holds. Suppose $J$  is generic and $g$ is a pluriclosed metric on $({\mathfrak g}, J)$. Then $\sigma b <0$. In particular, in this case  $({\mathfrak g}, J)$ cannot admit any balanced metric. 
\end{lemma}

\begin{proof}
Let us write $Y_{\alpha} = H_{\alpha} + S_{\alpha}$ where $H_{\alpha}$ is Hermitian symmetric and $S_{\alpha}$ is skew-Hermitian symmetric. Take $\alpha =\beta =1$ in equation (\ref{SKT2}), we get
\begin{equation} \label{HS}
[H_1, S_1] - 2S^2_1 + \overline{ D^1_{11}} \, S_1 + \overline{ D^1_{21}}  \, (S_2+tH_1-tS_1) \ = \ 0. 
\end{equation}
Since ${\mathfrak g}$ is unimodular and $b\neq 0$, by Lemma \ref{lemma7} and (\ref{abcd1}), (\ref{abcd2}), we have
$$ \mbox{tr} (S_2) = \frac{1}{2} \frac{i4\sigma}{\sqrt{2}\delta'}  = \frac{i\sqrt{2}\sigma }{\delta'}. $$
Since $Y_1$ is nilpotent hence traceless, by taking trace in (\ref{HS}) we get
$$ 2 \mbox{tr}(S_1^2) = \overline{ D^1_{21}} \, \mbox{tr}(S_2) = \frac{-ib\delta'}{\sqrt{2}} \frac{i\sqrt{2}\sigma }{\delta'} = \sigma b. $$ 
Thus $\sigma b = -2 \mbox{tr} (S_1S_1^{\ast}) \leq 0$, hence $<0$. This completes the proof of the lemma.
\end{proof}

Next let us turn our attention to the degenerate case. We prove the following:

\begin{lemma} \label{lemma12}
Let ${\mathfrak g}$ be a unimodular Lie algebra containing an abelian ideal ${\mathfrak a}$ of codimension $2$, and $J$ a complex structure on ${\mathfrak g}$. Assume that (\ref{assumption}) holds. Suppose $J$  is degenerate  and $g$ is a pluriclosed metric on $({\mathfrak g}, J)$. Then one must have $d'\neq 0$. In particular, in this case  $({\mathfrak g}, J)$ cannot admit any balanced metric. 
\end{lemma}

\begin{proof} 
Let us assume the contrary that $d'=0$. We want to derive at a contradiction. Under this assumption, the only possibly non-zero entries of $E_1$ and $E_2$ are
$$ D^1_{12} =-\frac{i\sigma}{\sqrt{2}\delta'} , \ \ \ \  \ \ \ D^2_{12} =\frac{ic' +\delta \sigma}{\sqrt{2}\delta'^2}. $$
Hence equation (\ref{HS}) become 
$$ [H_1, S_1] - 2S^2_1 = 0.$$
Taking trace in the above equation, we get $0=- 2\mbox{tr}(S_1^2)= 2\mbox{tr}(S_1S_1^{\ast})$. Hence $S_1=0$, so $Y_1=H_1$ is Hermitian and can be diagonalized. Alternatively, 
since $S_1$ is skew-Hermitian, by taking a unitary change of  $\{ e_3, \ldots , e_n\}$ if necessary, we may assume that $S_1$ is diagonal, say
$$ S_1 = \left[ \begin{array}{lll} i\lambda_1I_{n_1} &  & \\ & \ddots & \\ & & i\lambda_rI_{n_r} \end{array} \right], $$
where $\lambda_1, \ldots , \lambda_r$ are distinct real constants, and $n_1+\cdots +n_r=n-2$. Now $[H_1,S_1]=2S_1^2$ is diagonal so $H_1$ must be block-diagonal, and by unitary change of basis in each block, we may assume that both $S_1$ and $H_1$ are diagonal, so $Y_1$ is diagonal. 

On the other hand, by the condition $[Y_1, Y_2]=\frac{i\sigma}{\sqrt{2}\delta'} Y_1$ and $\sigma \neq 0$, we  know by Lemma \ref{lemma2} that $Y_1$ must be nilpotent. Therefore we conclude that $Y_1=0$.  

Now the only non-trivial case of equation (\ref{SKT2}) occur when $\alpha=\beta =2$ and in this case the equation becomes 
$$ [H_2, S_2] - 2S^2_2 = 0.$$
Taking trace, we get $0 = 2\mbox{tr}(S_2^2) = - 2\mbox{tr}(S_2S_2^{\ast})$, hence $S_2=0$. On the other hand, since $c=d'=0$, by Lemma \ref{lemma7} we get
$$ \mbox{tr}(S_2) = \frac{i\sigma}{\sqrt{2}\delta'} \neq 0, $$
which is a contradiction. This shows that the constant $d'$ cannot be zero, namely, the existence of a pluriclosed metric on $({\mathfrak g}, J)$ with $J$ degenerate would force $[{\mathfrak g}, {\mathfrak a}] \not\subseteq {\mathfrak b}$. This completes the proof of the lemma.
\end{proof}

Now we are ready to prove Theorem \ref{thm2}.

\begin{proof}[{\bf Proof of Theorem \ref{thm2} when ${\mathfrak g}/{\mathfrak a}$ is non-abelian}] 
Given a unimodular Lie algebra ${\mathfrak g}$ which contains an abelian ideal ${\mathfrak a}$ of codimension $2$, let $J$ be a complex structure on ${\mathfrak g}$ with $J{\mathfrak a} \neq {\mathfrak a}$. If ${\mathfrak g}/{\mathfrak a}$ is not abelian, then $J$ must fall into one and exactly one of the three cases: generic, half-generic, or degenerate. By Lemma \ref{lemma8}, the presence of a balanced metric on $({\mathfrak g}, J)$ would rule out the half-generic case, and in the generic case $\sigma b>0$ while in the degenerate case $d'=0$. On the other hand, by Lemma \ref{lemma11} and Lemma \ref{lemma12}, we know that in either of these cases the presence of a pluriclosed metric would lead to a contradiction. So for any given $J$ on ${\mathfrak g}$, it cannot admit both a balanced metric and a pluriclosed metric simultaneously. This establishes the Fino-Vezzoni Conjecture for the `$J{\mathfrak a}\neq {\mathfrak a}$ and ${\mathfrak g}/{\mathfrak a}$ non-abelian' case. The  `$J{\mathfrak a}\neq {\mathfrak a}$ and ${\mathfrak g}/{\mathfrak a}$ abelian' case is analogous, and will be given in Appendix B, while the  `$J{\mathfrak a}={\mathfrak a}$' case was due to \cite{LiZ} and will be outlined in Appendix A. Putting all three parts together, we have established the proof of Theorem \ref{thm2} and Corollary \ref{cor3}. 
\end{proof}

\vspace{0.3cm}
\section{Appendix A: The $J{\mathfrak a} = {\mathfrak a}$ case}

The proof of Theorem \ref{thm1} was given in \cite{LiZ}. However, the reference is in Chinese, and we will include the proof here in this appendix for readers convenience.

Throughout this appendix,  ${\mathfrak g}$ will be a unimodular Lie algebra containing an abelian ideal ${\mathfrak a}$ of codimension $2$, and $J$ will be a complex structure on ${\mathfrak g}$ satisfying $J{\mathfrak a} = {\mathfrak a}$. Suppose  $g = \langle , \rangle $ is a Hermitian metric on $({\mathfrak g},J)$, namely an inner product on ${\mathfrak g}$ compatible with $J$. 

Since $J{\mathfrak a} = {\mathfrak a}$, as in \cite{GuoZ} we have {\em admissible frames}, which are unitary frame $e=\{ e_1, \ldots , e_n\}$ such that 
$$ {\mathfrak a} = \mbox{span}_{\mathbb R} \{ e_i+\overline{e}_i, \sqrt{-1}(e_i-\overline{e}_i); \ 2\leq i\leq n \}. $$
Under such a frame, the only possibly none-zero structure constants $C$ and $D$ are given by
$$ C^j_{1i}=X_{ij}, \ \ \ D^1_{11}=\lambda, \
 \ \ D^j_{i1}=Y_{ij}, \ \ \ \ D^1_{ij}=Z_{ij}, \ \ \ D^1_{i1}=v_i; \ \  \ \ \ \  1\leq i,j\leq n, $$
where $\lambda \geq 0$, $v\in {\mathbb C}^{n-1}$ is a column vector, and  $X$, $Y$, $Z$ are $(n-1)\times (n-1)$ complex matrices. The Jacobi identities become
\begin{equation}  \label{Jacobi1}
\left\{ \begin{array}{ll}  \lambda (X^{\ast} +Y) + [X^{\ast} ,Y] - Z\overline{Z} \, =  \, 0, & \\
\lambda Z - (Z\,^t\!X+YZ) \ = \ 0. &
\end{array} \right.
\end{equation} 
Here $X^{\ast}$ means the conjugate transpose of $X$. By \cite{GuoZ}, we have
\begin{equation} \label{character}
\left\{ \begin{array}{llll} {\mathfrak g} \ \mbox{unimodular} \ \Longleftrightarrow \  \lambda - \mbox{tr}(X) + \mbox{tr}(Y)  = 0 , &&&\\
{\mathfrak g} \  \mbox{K\"ahler} \ \Longleftrightarrow \  v=0, \ X=Y, \ ^t\!Z=Z , &&&\\
{\mathfrak g} \  \mbox{balanced} \ \Longleftrightarrow \  v=0, \ \mbox{tr}(X) =\mbox{tr}(Y)   , &&&\\
{\mathfrak g} \  \mbox{pluriclosed} \ \Longleftrightarrow \  \lambda B + B^{\ast}B + [X^{\ast},B] + \,^t\!Z\overline{Z} - Z \overline{Z}  = 0 , &&&
 \end{array}\right.
\end{equation}
where $B=Y-X$. We have

\begin{lemma} \label{lemmaA}
Let ${\mathfrak g}$ be a unimodular Lie algebra with abelian ideal ${\mathfrak a}$ of codimension $2$, and $J$ be a complex structure on ${\mathfrak g}$ saisfying $J{\mathfrak a}={\mathfrak a}$. If $g$, $\tilde{g}$ are $J$-compatible inner products on ${\mathfrak g}$, then there exist  $g$-admissible frame $e$ and $\tilde{g}$-admissible frame $\tilde{e}$ such that
$$ \tilde{e}_1 = pe_1+\sum_{k=2}^n a_ke_k, \ \ \ \tilde{e}_i=p_i e_i, \ \ \  2\leq i\leq n, $$
where $p, p_2, \ldots , p_n$ are positive constants, and $a_2, \ldots , a_n$ are complex constants. Furthermore, under these frames, the corresponding structure constants are related by
\begin{eqnarray}
&& \tilde{\lambda} = p\lambda, \ \ \ \tilde{X}=pPXP^{-1}, \ \ \ \tilde{Y}=pP^{-1}YP, \ \ \ \tilde{Z}=pPZP^{-1}, \label{eq:A2}\\
&& \tilde{v}=pP^{-1}(pv-\lambda \overline{a} +Y\overline{a} + Za ), \label{eq:A3}
\end{eqnarray}
where $P=\mbox{diag}\{ p_2, \ldots , p_n\}$ and  $a=\,^t\!(a_2, \ldots , a_n)$. 
\end{lemma}

\begin{proof}
Let $e$ be an admissible frame for $g$. Then under appropriate unitary change of $\{ e_2, \ldots , e_n\}$, there are positive constants $p_2, \ldots ,p_n$ such that  $\{ p_2e_2, \ldots , p_ne_n\}$ is $\tilde{g}$-unitary. Extend it into a unitary frame $\tilde{e}$ for $\tilde{g}$, then its first element becomes 
$$ \tilde{e}_1 = pe_1 + \sum_{k=2}^n a_k e_k $$
with $p\neq 0$. Rotate  $\tilde{e}_1$ appropriately, we may assume that $p>0$. To see the relation between the two sets of structure constants, notice that, if we denote by $\{ \varphi_1, \ldots , \varphi_n\}$ the coframe dual to $e$, and write $\varphi$ for the column vector $\,^t\!(\varphi_2, \ldots, \varphi_n)$. Then the structure equation for the Hermitian metric $g$ becomes
\begin{equation} \label{eq:structureA}
\left\{ \begin{array}{ll} d\varphi_1  = \, - \lambda \varphi_1\overline{\varphi}_1 ,& \\
d\varphi \ = \  -  \varphi_1\overline{\varphi}_1\overline{v} - \varphi_1 \,^t\!X\varphi + \overline{\varphi}_1 \overline{Y}\varphi - \varphi_1 \overline{Z}\overline{\varphi}. & \end{array}\right.  
\end{equation}
Now for our admissible frames constructed before, their dual coframes are related by 
$$ \tilde{\varphi}_1=\frac{1}{p}\varphi_1, \ \ \ \tilde{\varphi}_i=\frac{1}{p_i}(\varphi_i-a_i\varphi_1), \ \ \ 2\leq i\leq n. $$
So by the structure equation (\ref{eq:structureA}) and its counterpart for $\tilde{g}$ we get the formula (\ref{eq:A2}) and (\ref{eq:A3}) relating the corresponding structure constants. This completes the proof of the lemma.
\end{proof}

\begin{proof}[{\bf Proof of Theorem \ref{thm1}.}] \  Assume that $({\mathfrak g},J)$ is as in Theorem \ref{thm1} and $g$, $\tilde{g}$ are Hermitian metrics on it so that $g$ is pluriclosed and $\tilde{g}$ is balanced. Fix a $g$-admissible frame $e$ and a $\tilde{g}$-admissible frame $\tilde{e}$ as in Lemma \ref{lemmaA}.

Since $\tilde{g}$ is balanced, we have $\mbox{tr}(\tilde{X})=\mbox{tr}(\tilde{Y})$. But ${\mathfrak g}$ is unimodular, thus $\tilde{\lambda}=\mbox{tr}(\tilde{X})-\mbox{tr}(\tilde{Y})=0$. By the first equality of (\ref{eq:A2}) we get $\lambda =0$. From (\ref{Jacobi1}) we deduce  $\mbox{tr}(Z\overline{Z})=0$. Now by taking trace in the last equation of (\ref{character}) we get
$$ \mbox{tr} (B^{\ast}B) + \mbox{tr} (\,^t\!Z \overline{Z}) = 0.$$
This gives us $Z=0$ and $B=0$, which is $X=Y$. 

Next, again since $\tilde{g}$ is balanced, we have $\tilde{v}=0$. So by (\ref{eq:A3}) we get $v=-\frac{1}{p}Y\overline{a}$. In summary, our pluriclosed metric $g$ now has structure constants:
$\lambda =0$, $Z=0$, $X=Y$ which is normal by (\ref{Jacobi1}), and with $v$ in the range of $X$. This metric $g$ in general might not be K\"ahler as $v$ may be  non-zero. Let us take
$$ e_1'=pe_1+\sum_{k=2}^n a_ke_k, \ \ \ e_i'=e_i, \ \ \ 2\leq i \leq n.$$
Let $g'$ be the Hermitian metric on $({\mathfrak g},J)$ so that $e'$ becomes unitary. Then by Lemma \ref{lemmaA} the structure constants for $g'$ are given by
$$ \lambda'=p\lambda =0, \ \ \ X'=pX, \ \ \ Y'=pY, \ \ \ Z'=pZ=0, \ \ \ v'=p(pv-Y\overline{a})=0. $$
In particular, $X'=Y'$, $\,^t\!Z'=Z'$, and $v'=0$. So the metric $g'$ is K\"ahler. This completes the proof of Theorem \ref{thm1}. 
\end{proof}

\vspace{0.3cm}
\section{Appendix B: The $J{\mathfrak a} \neq {\mathfrak a}$ and ${\mathfrak g} / {\mathfrak a}$ abelian case}

Throughout this appendix, we will assume that {\em  ${\mathfrak g}$ is a unimodular Lie algebra of real dimension $2n$, ${\mathfrak a}\subseteq {\mathfrak g}$ an abelian ideal of codimension $2$ with ${\mathfrak g} / {\mathfrak a}$  abelian, and $J$ a complex structure on ${\mathfrak g}$ with $J{\mathfrak a} \neq {\mathfrak a}$.} 

Write ${\mathfrak g}'=[{\mathfrak g},{\mathfrak g}]$. Since ${\mathfrak g} / {\mathfrak a}$ is abelian, we have ${\mathfrak g}'\subseteq {\mathfrak a}$ so ${\mathfrak g}$ is $2$-step solvable. In \cite{FSwann22} and \cite{FSwann}, Freibert and Swann systematically studied Hermitian structures on all $2$-step solvable Lie algebras, and in particular they proved the following:

\begin{theorem}[{\bf Freibert-Swann}]
Given any $2$-step solvable Lie algebra ${\mathfrak g}$ and any complex structure $J$ on ${\mathfrak g}$,  then the Fino-Vezzoni Conjecture holds if $J$ belongs to one  of the following three pure types:
\begin{itemize}
\item  Pure Type I: \ \, $J{\mathfrak g}'\cap {\mathfrak g}'=0$;
\item  Pure Type II: \, $J{\mathfrak g}'={\mathfrak g}'$;
\item  Pure Type III: $J{\mathfrak g}'+{\mathfrak g}'={\mathfrak g}$.
\end{itemize}
\end{theorem}

In our case $J$ might not of pure types, so we cannot directly applying their result here, even though their method is much more powerful and is dealing with a much larger class of Lie algebras. Write ${\mathfrak a}_J={\mathfrak a} \cap J{\mathfrak a}$ as before. Let us introduce the following:

\begin{definition}[{\bf The rank of $({\mathfrak g},J)$}]
It is the integer: $r_0=\dim_{\mathbb R}\{({\mathfrak g}'+{\mathfrak a}_J)/{\mathfrak a}_J\} \in \{ 0,1,2\} $. 
\end{definition}

Take any $V\cong {\mathbb R}^2$ so that ${\mathfrak a}={\mathfrak a}_J \oplus V$ and any basis $\{ x,y\}$ in $V$. Then ${\mathfrak g}={\mathfrak a} \oplus JV$, and all the non-trivial Lie bracket info are given by $Jx$ and $Jy$. Since $[x,y]=0$, the integrability condition for $J$ gives us
$$ [Jx,Jy] = J([Jx,y]-[Jy,x]) \in {\mathfrak a}_J. $$ 
As in (\ref{abcd}), let us write

\begin{equation} \label{abcdA}
\left\{ \begin{split} 
[Jx, x] =  ax + by + {\mathfrak a}_J ;\ \,\\ 
[Jx, y] =  cx + dy + {\mathfrak a}_J ;\ \,\\ 
[Jy, x] =  cx + dy + {\mathfrak a}_J ; \ \,\\
\ [Jy, y] =  c'x + d'y + {\mathfrak a}_J,
\end{split}
\right.
\end{equation}
where $a,b,c,d,c',d'$ are real constants, with matrices
$$ A_x=\left[ \begin{array}{ll} a & b \\ c & d \end{array} \right] \ \ \ \mbox{and} \ \ \ A_y=\left[ \begin{array}{ll} c & d \\ c' & d' \end{array} \right] $$
representing the induced action of $\mbox{ad}_{Jx}$ and $\mbox{ad}_{Jy}$ on $V\cong {\mathfrak a}/{\mathfrak a}_J$. Clearly, $[A_x, A_y]=0$. Replace $\{ x,y\}$ by another basis of $V$ if necessary, we may always assume that
\begin{equation} \label{traceA}
\mbox{tr}(A_x) =0.
\end{equation}
Namely, we have $d=-a$ from now on. Note that this is not the most convenient choice, but we made it here simply to be in consistence with \S 3 so the computations can be carried over. 

\begin{remark}
When $r_0=2$, the choice (\ref{traceA}) will uniquely determine $x\in V$  up to a constant multiple. In particular, the sign of $\,\det (A_x)$ depends only on $J$.
\end{remark}

This is because ${\mathfrak h}=\mbox{span}\{ A_x, A_y\}$ forms a $2$-dimensional abelian Lie subalgebra in ${\mathfrak g}{\mathfrak l}(2,{\mathbb R})$, so it must  contain a basis of the form $\{ I, X\}$ with $\mbox{tr}(X)=0$. The commutativity $[A_x,A_y]=0$ becomes
\begin{equation} \label{abcdB}
\left\{ \begin{split} \,ac+ bc' = 0; \ \ \ \ \ \, \\
\ b(d' -c)= 2a^2 ;\ \ \,\\
 c (d'-c) = -2ac' . 
\end{split} \right.
\end{equation}
This is parallel to (\ref{abcd2}), except now our $\sigma =0$ while $A_x$ might not be nilpotent, in other words $a^2+bc$ may not be $0$ anymore. 

Next, let $g=\langle , \rangle $ be an inner product on ${\mathfrak g}$ compatible with $J$. Let us choose $V$ to be the $g$-orthogonal complement of ${\mathfrak a}_J$ in ${\mathfrak a}$ and $\{ x,y\}$ be an orthonormal basis in $V$, with (\ref{traceA}) satisfied. Note that when $r_0=2$ this will uniquely determine $x$ (up to $\pm$ sign). Let $\delta =\langle Jx, y\rangle \in (-1,1)$ as before, which measures the angle between $V$ and $JV$, and let $\delta'=\sqrt{1-\delta^2} \in (0,1]$.  

\begin{definition} [{\bf Admissible frames for $({\mathfrak g},J,g)$}] A unitary basis $e$ of ${\mathfrak g}^{1,0}$ is called an admissible frame for ${\mathfrak g}$, if   $\{ e_3, \ldots , e_n\}$ spans ${\mathfrak a}_J^{1,0}$, $e_1=\frac{1}{\sqrt{2}}(x-iJx)$, and $Y=i\delta e_1+\delta'e_2$ where $Y=\frac{1}{\sqrt{2}}(y-iJy)$. Here $\{ x,y\}$ is an orthonormal basis of $\,V:={\mathfrak a}_J^{\perp} \cap {\mathfrak a}$ with $\mbox{tr}(A_x)=0$.
\end{definition}

Note that the situation here is strictly analogous to \S 3, with the only exception that we have $\sigma =0$ here but possibly $a^2+bc\neq 0$, as $A_x$ may no longer be nilpotent. Under an admissible frame $e$, Lemma \ref{lemma4} still holds, and we have $C^1_{12}=C^2_{12}=0$. Also, the formula (\ref{D1}), (\ref{D2}), (\ref{D3}), (\ref{D4}) are still valid (with $\sigma =0$). In summary, we have the following:

\begin{lemma} \label{lemmaB}
Let $J$ be a complex structure on a Lie algebra ${\mathfrak g}$ which contains   an abelian ideal ${\mathfrak a}$ of codimension $2$. Assume that $J{\mathfrak a}\neq {\mathfrak a}$ and ${\mathfrak g}/{\mathfrak a}$ is abelian. If $e$ is an admissible frame, then the structure constants $C$ and $D$ satisfy the conclusion of Lemma \ref{lemma4}, with $D$ given by (\ref{11}) and 
\begin{equation} \label{eq:lemmaB}
\left\{ \begin{array}{llll}
 D^1_{11}=\frac{1}{\sqrt{2}} (b\delta +ia), \ \ \  D^1_{21}=\frac{1}{\sqrt{2}} ib\delta', &&& \\
 D^2_{11}=\frac{1}{\sqrt{2}\delta'} \{ -2a\delta +i(c +b\delta^2)\}, \ \ \  D^2_{21}=\frac{1}{\sqrt{2}} (-b\delta -ia), &&& \\
 D^1_{12}=\frac{1}{\sqrt{2}\delta’} i(c -b\delta^2), \ \ \  D^1_{22}=\frac{1}{\sqrt{2}} (b\delta -ia), &&& \\
 D^2_{12}=\frac{1}{\sqrt{2}\delta'^2} \{ \delta(d'+b\delta^2) +i(c'+a\delta^2)\}, \ \ \  D^2_{22}=\frac{1}{\sqrt{2}\delta'} i(d'+b\delta^2). &&&
\end{array}\right.
\end{equation}
\end{lemma}
For convenience of our later discussions, let us introduce the notation:
$$ Z_1=Y_1-Y_1^{\ast}, \ \ \ \ \ \ Z_2=Y_2-Y_2^{\ast} +2t Y_1^{\ast}, $$
where $t=i\delta /\delta'$. By (\ref{unimodular}), Lemma \ref{lemma4} and (\ref{eq:lemmaB}), we see that:  
\begin{equation} \label{unimodularB}
{\mathfrak g} \ \,\mbox{is unimodular }  \ \ \Longleftrightarrow \ \ \, \mbox{tr}(Z_1)= \, \mbox{tr}(Z_2)+\frac{1}{\sqrt{2}\delta'}i(c+d')=0.
\end{equation}
Also, under the assumption that ${\mathfrak g}$ is unimodular, the metric $g$ will be balanced if and only if $\sum_{s=1}^n D^s_{as}=0$ for all $1\leq a\leq n$. From  Lemma \ref{lemma4} and (\ref{eq:lemmaB}), we see that this is equivalent to
$$ v^1_1+v^2_2=0, \ \ \ a+c'=0, \ \ \ b+d'=0. $$
Recall that $v^{\alpha}_{\beta}= (D^{\alpha}_{i\beta})_{\{3\leq i\leq n\}}$. Plug into (\ref{abcdB}), we get $a=c'=0$ and $c=d'=-b$. In summary, we have

\begin{lemma} \label{lemmaB2}
Let $({\mathfrak g}, J)$ be as in Lemma \ref{lemmaB}, with  ${\mathfrak g}$ unimodular. Then a Hermitian metric $g$ on  $({\mathfrak g}, J)$ is balanced if and only if
\begin{equation*}
v^1_1+v^2_2=0, \ \ \  A_x = \left[ \begin{array}{cc} 0 & b \\ -b & 0 \end{array}\right], \ \ \ A_y = \left[ \begin{array}{cc} -b & 0 \\ 0 & -b  \end{array}\right]
\end{equation*}
In particular, the rank $r_0$ of $({\mathfrak g},J)$ is either $0$ or $2$, depending on whether $b=0$ or not. Also, when $r_0=2$, the sign $\det (A_x) >0$. 
\end{lemma} 

Similarly, the metric $g$ is K\"ahler if Chern torsion vanishes, namely, $T^b_{ac}=-C^b_{ac}-D^b_{ac}+D^b_{ca} =0$ for all $1\leq a,b,c\leq n$. By Lemma \ref{lemma4}, (\ref{abcdB}) and (\ref{eq:lemmaB}), we deduce that

\begin{lemma} \label{lemmaB3}
Let $({\mathfrak g}, J)$ be as in Lemma \ref{lemmaB}, with  ${\mathfrak g}$ unimodular. Then a Hermitian metric $g$ on  $({\mathfrak g}, J)$ is K\"ahler if and only if
\begin{equation*}
Z_1=Z_2=0, \ \ \ \ r_0=0, \ \ \ \ \mbox{and} \ \ \ \ v^{\alpha}_{\beta} =0 \ \ \forall \ 1\leq \alpha, \beta \leq 2.
\end{equation*}
\end{lemma} 

Notice that $r_0=0$ is equivalent to $a=b=c=c'=d'=0$ here. Our next goal would be to get the condition for pluriclosedness, which is characterized by the equation (\ref{SKT}). As in \S 5, by letting $k=\alpha$ and $\ell =\beta$ be in $\{ 1,2\}$ and $i,j\in \{ 3, \ldots , n\}$ in (\ref{SKT}), we end up with (\ref{SKT2}), which in our notation $Z_1$, $Z_2$ is just
\begin{equation} \label{SKTB}
-Z_{\alpha}Z_{\beta} + [Y^{\ast}_{\beta}, Z_{\alpha}] + \overline{D^{\alpha}_{1\beta} }Z_1 +  \overline{ D^{\alpha}_{2\beta} } Z_2 =0, \ \ \ \ 1\leq \alpha , \beta \leq 2. 
\end{equation}
Letting $\alpha =\beta =1$, taking trace, and utilizing (\ref{unimodularB}), we get
$$  -\mbox{tr}(Z_1Z_1) - \overline{D^1_{21} } \frac{1}{\sqrt{2}\delta'} i(c+d') =0. $$
Therefore by (\ref{lemmaB}) we get
$$ \mbox{tr}(Z_1Z_1^{\ast}) = -\mbox{tr}(Z_1Z_1) = \overline{D^1_{21} } \frac{1}{\sqrt{2}\delta'} i(c+d') = -\frac{1}{\sqrt{2}} ib\delta'\frac{1}{\sqrt{2}\delta'} i(c+d') = \frac{1}{2} b(c+d'). $$
By the middle line of (\ref{abcdB}), which is $b(d'-c)=2a^2$, we obtain
$$ -2\det (A_x) = 2bc + 2a^2 = 2bc + b(d'-c) = b(c+ d') = 2\mbox{tr}(Z_1Z_1^{\ast}) \geq 0,$$
therefore we  get the following:

\begin{lemma} \label{lemmaB4}
Let $({\mathfrak g}, J)$ be as in Lemma \ref{lemmaB}, with  ${\mathfrak g}$ unimodular. If a Hermitian metric $g$ on  $({\mathfrak g}, J)$ is pluriclosed, then
\begin{equation*}
\det (A_x) \leq 0 
\end{equation*}
for any $x\in {\mathfrak a} \setminus {\mathfrak a}_J$ with $\mbox{tr}(A_x)=0$. 
\end{lemma}

Now suppose that $J{\mathfrak a}\neq {\mathfrak a}$ and ${\mathfrak g}/{\mathfrak a}$ is abelian. Assume that $({\mathfrak g},J)$ admits  both a balanced metric and a pluriclosed metric. By Lemma \ref{lemmaB2} we know that we must have $r_0=0$ or $r_0=2$, and in the $r_0=2$ case we must have $\det (A_x)>0$. On the other hand, the presence of a pluriclosed metric would imply $\det (A_x)\leq 0$ by  Lemma \ref{lemmaB4}. By Remark 2, we know that when $r_0=2$ the sign of $\det (A_x)$ is uniquely determined by $J$, thus the same $J$ cannot simultaneously admit both types of metrics. So in order to verify the Fino-Vezzoni Conjecture, we may restrict ourselves to the $r_0=0$ case. 

{\em For now on, we will assume that $r_0=0$, which means $D^{\beta}_{\alpha \gamma }=0$ for all $1\leq \alpha , \beta , \gamma \leq 2$.} In this case, the only possibly non-zero $D$ components  are 
$$Y_1=(D^j_{i1}), \ \ \ Y_2=(D^j_{i2}), \ \ \ v^{\alpha}_{\beta}=(D^{\alpha}_{i\beta}), \ \ \ \ 1\leq \alpha , \beta \leq 2; \ \ \ 3\leq i,j\leq n, $$
satisfying
\begin{equation} \label{r0=0}
[Y_1,Y_2]=0, \ \ \ \ \ Y_1v^{\alpha}_2=Y_2v^{\alpha}_1 \ \ \ (1\leq \alpha \leq 2),
\end{equation}
and the only possibly non-zero $C$ components are
$$ C_{\alpha} =(C^j_{\alpha i})= Y_{\alpha}^{\ast} - 2t \delta_{\alpha 2} Y_1^{\ast}, \ \ \ \ (C^i_{12})=\overline{v^2_1}- \overline{v^1_2} + 2t \overline{v^1_1}. $$

Now let us assume that the metric $g$ is pluriclosed. Since all $D^{\beta}_{\alpha \gamma }=0$, by  taking trace in (\ref{SKTB}), we get
$$ \mbox{tr} ( Z_{\alpha} Z_{\beta})=0. $$
In particular, $\mbox{tr} ( Z_1Z_1^{\ast})= -\mbox{tr} ( Z_1Z_1)=0$ leads to $Z_1=0$. So $Y_1=Y_1^{\ast}$ is Hermitian symmetric, hence $Z_2^{\ast}=-Z_2$ as $t=i\delta /\delta'$ is pure imaginary. This leads to $Z_2=0$ as well. 

In order to prove the Fino-Vezzoni Conjecture for this $r_0=0$ case, we will need to squeeze out more information from (\ref{SKT}). By letting $i=k=1$ and $j=\ell =2$ in (\ref{SKT}), we obtain
\begin{equation} \label{SKTv}
|v^2_1|^2 + |v^1_2|^2 + |v^2_1-v^1_2-2t v^1_1|^2 - 2\mbox{Re} \langle v^1_1, \overline{v^2_2} \rangle = 0. 
\end{equation}
Also, if we let $j=1$, $\ell =2$, $k=\alpha \in \{ 1,2\}$ and $3\leq i\leq n$ in (\ref{SKT}), we get
\begin{equation} \label{SKTvY}
Y_1^{\ast} v^2_{\alpha} = Y^{\ast}_2 v^1_{\alpha} \ \ \ \ (1\leq \alpha \leq 2). 
\end{equation}
In summary, we have the following:

\begin{lemma} \label{lemmaB5}
Let $({\mathfrak g}, J)$ be as in Lemma \ref{lemmaB}, with  ${\mathfrak g}$ unimodular and the rank of $({\mathfrak g}, J)$ be  $r_0=0$. If a Hermitian metric $g$ on  $({\mathfrak g}, J)$ is pluriclosed, then $Z_1=Z_2=0$, and the equations (\ref{r0=0}), (\ref{SKTv}) and (\ref{SKTvY})
hold. 
\end{lemma} 

Note that in this case the pluriclosed metric $g$ may not be K\"ahler, as the $v$ part might not vanish. Analogous to the situation in Theorem \ref{thm1}, we will show that if $({\mathfrak g}, J)$ admits another metric $\tilde{g}$ which is balanced, then those $v^{\alpha}_{\beta}$ for the pluriclosed metric $g$ will fall inside the combined range of $Y_1$ and $Y_2$, thus by an appropriate modification of $g$ we will be able to produce a K\"ahler metric on $({\mathfrak g}, J)$. To achieve that goal, we will need to sort out the relationship between the two sets of strucutre constants for any two given metrics on  $({\mathfrak g}, J)$.

Let $({\mathfrak g}, J)$ be as in Lemma \ref{lemmaB} with ${\mathfrak g}$ unimodular. Assume that $r_0=0$ and $g$, $\tilde{g}$ are Hermitian metrics on $({\mathfrak g}, J)$. We claim that there exist $g$-admissible frame $e$ and $\tilde{g}$-admissible frame $\tilde{e}$ so that
\begin{equation} \label{framechange}
 \tilde{e}_1= p_1e_1 + \sum_{k=3}^n a_k e_k, \ \ \   \tilde{e}_2= \mu e_1+ p_2e_2 + \sum_{k=3}^n b_k e_k, \ \ \  \tilde{e}_i= p_ie_i  \ \ \ (3\leq i\leq n),
\end{equation}
where $p_1, \ldots , p_n$ are positive constants, and $\mu$, $a_k$, $b_k$ are complex constants. 

To see this, let us start with any $g$-admissible frame $e$. By an appropriate unitary change of $\{ e_3, \ldots , e_n\}$, we may assume that $\{ p_3e_3, \ldots , p_n e_n\}$ is $\tilde{g}$-unitary for some positive constants $p_3, \ldots , p_n$. On the other hand, since $r_0=0$, any orthonormal basis $\{ x,y\}$ of the orthogonal complement $V={\mathfrak a}_J^{\perp} \cap {\mathfrak a}$ will work in forming an admissible frame. In particular, we may choose $\{ \tilde{x}, \tilde{y}\}$ so that $\tilde{x}$ lies in ${\mathbb R}x+{\mathfrak a}_J$. Thus the resulting  admissible frames will be related by the above equations. Using (\ref{framechange}) and the definition
$$[\overline{e}_i,e_j] = \sum_{k=1}^n \{ D^i_{kj} \overline{e}_k -  \overline{D^j_{ki} } e_k \} \ \ \ \ (1\leq i,j\leq n), $$
a straight-forward computation yields the following:
\begin{equation} \label{framechange2}
\left\{ \begin{array}{llllll}  \tilde{Y}_1 = P^{-1}(p_1Y_1)P, &&&&& \\
\tilde{Y}_2 = P^{-1}(\mu Y_1+p_2Y_2)P, &&&&& \\
\tilde{v}^1_1 = P^{-1}(p_1^2v^1_1 + p_1Y_1\overline{a}), &&&&&\\
\tilde{v}^2_1 = P^{-1}(p_1p_2v^2_1 + p_1\overline{\mu} v^1_1 - p_1Y_1\overline{b}), &&&&&\\
\tilde{v}^1_2 = P^{-1}\{ p_1p_2v^1_2 + p_1\mu  v^1_1 + (\overline{\mu }Y_1 + p_2Y_2) \overline{a}\}, &&&&&\\
\tilde{v}^2_2 = P^{-1}\{ p_2^2 v^2_2 + p_2\mu  v^2_1 + p_2\overline{\mu }v^1_2 + |\mu |^2 v^1_1  + (\mu Y_1+p_2Y_2) \overline{b}\} ,
\end{array} \right.
\end{equation}
where $P=\mbox{diag}\{ p_3, \ldots , p_n\}$, $a=\,^t\!(a_3, \ldots , a_n)$, $b=\,^t\!(b_3, \ldots , b_n)$. 

Let us write $U={\mathfrak a}_J\otimes {\mathbb C} \cong {\mathbb C}^{n-2}$. Since $Y_1$ is Hermitian symmetric, it can be diagonalized. Also, both $Y_2$ and $Y_2^{\ast}$ commute with $Y_1$ and $Z_2=0$, thus $Y_1$ and $Y_2$ can be simultaneously diagonalized. Let us denote the  sum space of the ranges of $Y_1$ and $Y_2$ by $W\subseteq U$, and its orthogonal complement in $U$ by $W^{\perp}$. By a generic $SO(2)$ change of $\{ x,y\}$, we may assume that both $Y_1$ and $Y_2$ are of rank $s=\dim_{\mathbb C}(W)$. So $W^{\perp}$ is the kernel space for both $Y_1$ and $Y_2$. 

For any vector $w\in U$, we will write $w=w'+w^{\perp}$ for the decomposition in $U=W\oplus W^{\perp}$. Since $Y_1$, $Y_2$ are non-degenerate on $W$, by (\ref{r0=0}) and (\ref{SKTvY}), we get in $W$ that
$$ (v^1_2)' = A\xi , \ \ \ \ \ (v^2_1)' =A^{\ast} \xi , \ \ \ \ \  (v^2_2)' = AA^{\ast}\xi  ,$$
where $\xi = (v^1_1)'$ and $A=(Y_1|_W)^{-1}(Y_2|_W)$.  Since $Z_1=Z_2=0$, we have $Y_1^{\ast}=Y_1$ and $Y_2^{\ast} = Y_2+2tY_1$, hence $Y_2Y_2^{\ast} = Y_2^{\ast} Y_2$, so $AA^{\ast} = A^{\ast}A$. Therefore
$$ 2\mbox{Re} \langle (v^1_1)', \overline{(v^2_2)'} \rangle  = 2|A\xi |^2 = 2|A^{\ast} \xi |^2. $$
This means that in the formula (\ref{SKTv}), the $W$-portion of the left hand side  is non-negative, thus the $W^{\perp}$-portion of the left hand side needs to be non-positive, yielding
\begin{equation} \label{perp}
|(v^2_1)^{\perp}|^2 + |(v^1_2)^{\perp}|^2 - 2\mbox{Re} \langle (v^1_1)^{\perp}, \overline{(v^2_2)^{\perp}} \rangle  \leq 0. 
\end{equation}
Now since $\tilde{g}$ is a balanced metric, we have by Lemma \ref{lemmaB2} that $\tilde{v}^1_1+ \tilde{v}^2_2=0$. By (\ref{framechange2}), this means that
$$ p_1^2v^1_1 + p_1Y_1\overline{a} + p_2^2 v^2_2 + p_2\mu v^2_1 + p_2\overline{\mu} v^1_2 + |\mu |^2 v^1_1 + (\mu Y_1+p_2Y_2)\overline{b} = 0. $$
Taking the $W^{\perp}$-part, we get
$$ (v^2_2)^{\perp} = - \frac{p_1^2+|\mu |^2}{p_2^2} (v^1_1)^{\perp} - \frac{|\mu|}{p_2} (\rho (v^2_1)^{\perp} + \overline{\rho }(v^1_2)^{\perp}), $$
where we wrote $\mu = |\mu |\rho$. Write $w=\rho (v^2_1)^{\perp} + \overline{\rho }(v^1_2)^{\perp} $ and plug it into (\ref{perp}), we get
\begin{equation} \label{eq:ineq}
 |(v^2_1)^{\perp}|^2 + |(v^1_2)^{\perp}|^2 + 2\frac{p_1^2+|\mu|^2 }{p_2^2 }|(v^2_1)^{\perp}|^2 + 2\frac{|\mu|}{p_2}\mbox{Re} \langle (v^1_1)^{\perp}, \overline{w} \rangle  \leq 0 . 
\end{equation}
If $\mu=0$, the above leads to all $(v^{\alpha}_{\beta})^{\perp}=0$ already. Assume $\mu \neq 0$, replace the last term on the left hand side by 
$$ 2\mbox{Re} \langle (v^1_1)^{\perp}, \overline{w} \rangle \geq - \frac{2|\mu|}{p_2} |(v^1_1)^{\perp}|^2 - \frac{p_2}{2|\mu|} |w|^2, $$
the inequality (\ref{eq:ineq}) becomes 
$$  |(v^2_1)^{\perp}|^2 + |(v^1_2)^{\perp}|^2 - \frac{1}{2}|w|^2 + 2\frac{p_1^2}{p_2^2 }|(v^1_1)^{\perp}|^2  \leq 0. $$
Clearly this gives us $(v^1_1)^{\perp} = 0$, and by (\ref{eq:ineq}) again we get $(v^{\alpha}_{\beta})^{\perp}=0$ for all $\alpha, \beta \in \{ 1,2\}$. In summary, the presence of a balanced metric allows us to conclude that $(v^{\alpha}_{\beta})^{\perp}=0$ for all $1\leq \alpha, \beta \leq 2$. 

Finally, let us modify our pluriclosed metric $g$ into a K\"ahler metric. Let us choose $\mu =0$, $p_1=\cdots =p_n=1$,  and choose $a$ and $b$ so that
$$ v^1_1+Y_1\overline{a}=0, \ \ \ v^2_1-Y_1 \overline{b}=0. $$
In other words we form a new metric (which we still denote as $\tilde{g}$ for convenience) such that $\tilde{e}_1=e_1+\sum_k a_ke_k$, $\tilde{e}_2=e_2+\sum_kb_ke_k$, $\tilde{e}_i=e_i$ ($3\leq i\leq n$) form its unitary frame. By  the coefficient change formula (\ref{framechange2}), we get $\tilde{Y}_1=Y_1$, $\tilde{Y}_2=Y_2$, and 
$$  \tilde{v}^1_1 = \tilde{v}^2_1=0. $$
By the second part of (\ref{r0=0}), which is valid for any Hermitian metric on $({\mathfrak g},J)$, we know that $ \tilde{v}^1_2 = \tilde{v}^2_2=0$ as well. So the metric $\tilde{g}$ satisfies $\tilde{Z}_1=\tilde{Z}_2=0$ and $\tilde{v}=0$, hence is K\"ahler by Lemma \ref{lemmaB3}. Now we have completed the proof of Theorem \ref{thm2} in the $J{\mathfrak a}\neq {\mathfrak a}$ and ${\mathfrak g}/{\mathfrak a}$ abelian case.

\vspace{0.3cm}

\vs

\noindent\textbf{Acknowledgments.} The second named author would like to thank Bo Yang and Quanting Zhao for their interests and/or helpful discussions. We would also like to take this opportunity to thank the referee for a number of valuable corrections and suggestions, which improved the readability and completeness of the paper.

\vs

\end{document}